\documentclass{mrlart7} 
\def\volno{18}
\def\yrno{0000}
\def\issueno{00}
\def\lpageno{10032} %temporary numbers
\usepackage{mathrsfs}
\usepackage{amsfonts}
\usepackage{amsthm}
\usepackage{dsfont}
\usepackage{amssymb,latexsym,amsmath}
\usepackage{graphicx}
\usepackage{verbatim}
\usepackage{float}
\usepackage{caption}
\usepackage[all]{xy}
\usepackage{tikz}
\usepackage{hyperref}
\usepackage{tikz-cd}
\usetikzlibrary{matrix}

\newtheorem{theorem}{Theorem}[section]
\newtheorem{innercustomthm}{Theorem}
\newenvironment{customthm}[1]
  {\renewcommand\theinnercustomthm{#1}\innercustomthm}
  {\endinnercustomthm}
\newtheorem{innercustomDef}{Definition}
\newenvironment{customDef}[1]
  {\renewcommand\theinnercustomDef{#1}\innercustomDef}
  {\endinnercustomDef}
\newtheorem{definition}{Definition}[section]
\newtheorem{lemma}{Lemma}[section]
\newtheorem{proposition}{Proposition}[section]
\newtheorem{conjecture}{Conjecture}
\newtheorem*{remark}{Remark}

\newcommand\fnote[1]{\captionsetup{font=small}\caption*{#1}}

\begin{document}

\newcommand{\GT}{\widetilde{PSL_2\mathbb{R}}}
\newcommand{\G}{PSL_2\mathbb{R}}

\title{Orderability of Homology Spheres Obtained by Dehn Filling}
\author{Xinghua Gao}
\address{Department of Mathematics \\ University of Illinois at Urbana-Champaign \\ 1409 W. Green St, Urbana, IL, 61801}
\email{xgao29@126.com}

%\classification{57M60; 57M25, 57M05, 20F60}
%\keywords{left-orderable fundamental group, homology sphere, L-space Conjecture, character variety}

\begin{abstract}
In this paper, we develop a method for constructing left-orders on the fundamental groups of rational homology 3-spheres. We begin by constructing the holonomy extension locus of a rational homology solid torus $M$, which encodes the information about peripherally hyperbolic $\widetilde{\text{PSL}_2\mathbb{R}}$ representations of $\pi_1(M)$. Plots of the holonomy extension loci of many rational homology solid tori are shown, and the relation to left-orderability is hinted. Using holonomy extension loci, we study rational homology 3-spheres coming from Dehn filling on rational homology solid tori and construct intervals of Dehn fillings with left-orderable fundamental group.
\end{abstract}

\maketitle

\section{Introduction}
A nontrivial group is called left-orderable if there exists a strict total order on the set of group elements which is invariant under left multiplication. We will say that a closed 3-manifold is orderable when its fundamental group is left-orderable. In particular, for a countable group, being left-orderable is the same as being isomorphic to a subgroup of Homeo$^+(\mathbb{R})$, the group of orientation preserving homeomorphisms of $\mathbb{R}$ (see e.g. \cite[Theorem 2.6]{BRW}). 

The reason why we care about left-orderability is that this property is conjectured to detect L-spaces completely. Recall an irreducible $\mathbb{Q}$-homology 3-sphere (abbr. $\mathbb{Q}$HS) $Y$ is called an L-space if dim $\widehat{HF}(Y) = |H_1(Y; \mathbb{Z})|$, i.e. it obtains minimal Heegaard Floer homology \cite{lspaceknot}. Boyer, Gordon, and Watson conjectured in \cite{BGW} that a $\mathbb{Q}$HS is a non L-space if and only if its fundamental group is left-orderable. This conjecture has been studied extensively in recent years and evidence has accumulated in favor of the conjecture \cite{BC, GL14}.

One of the main difficulties of proving the conjecture is to determine the left-orderability of a fundamental group. Various tools have been developed to study left-orderability. Boyer, Rolfsen, and Wiest proved that a necessary and sufficient condition that the fundamental group of a compact, connected, orientable 3-manifold be left-orderable is that there exists a homomorphism into Homeo$^+(\mathbb{R})$ \cite[Theorem 3.2]{BRW}. In particular, representation (i.e. homeomorphism) of the fundamental group into $\widetilde{\text{PSL}_2\mathbb{R}}$, a subgroup of Homeo$^+(\mathbb{R})$,  has been proven very useful in studying the left-orderability of $3$-manifold groups \cite{CD16, EHN, Hu, Tran_cover, Tran_surg}. Being the universal cover of the linear Lie groups PSL$_2\mathbb{R}$ and SL$_2\mathbb{R}$, $\widetilde{\text{PSL}_2\mathbb{R}}$ is more computable as compared to Homeo$^+(\mathbb{R})$, while still containing a fair amount of information about left-orders inherited from Homeo$^+(\mathbb{R})$. 

To study $\widetilde{\text{PSL}_2\mathbb{R}}$ representations, Culler and Dunfield introduced the idea of the translation extension locus of a compact 3-manifold $M$ with torus boundary \cite{CD16}. They gave several criteria implying whole intervals of Dehn fillings of $M$ have left-orderable fundamental groups.

\subsection{The translation extension locus}
We follow the notation in \cite{CD16}. Denote PSL$_2\mathbb{R}$ by $G$, and $\widetilde{\text{PSL}_2\mathbb{R}}$ by $\widetilde{G}$. Let $R_{\widetilde{G}}(M)=\text{Hom}(\pi_1(M), \widetilde{G})$ be the variety of $\widetilde{G}$ representations of $\pi_1(M)$. For a precise definition of the representation variety, see Section \ref{variety}.

The name translation extension locus comes from the fact that we need to use translation number in the definition.
For an elements $\widetilde{g}$ in $\widetilde{G}$, define the translation number to be
\begin{displaymath}
 \text{trans}(\widetilde{g})=\lim_{n\to\infty}\frac{\widetilde{g}^n(x)-x}{n} \text{ for some } x\in \mathbb{R}.
\end{displaymath}
Then trans: $R_{\widetilde{G}}(\partial M)\rightarrow H^1(\partial M; \mathbb{R})$ can be defined by taking $\widetilde{\rho}$ to trans$\circ \widetilde{\rho}$.
The translation number of $\widetilde{g}$ measure the distance $\widetilde{g}$ moves a point on the real line. It will be discussed in more details in the next section.

Let $M$ be a knot complement in a $\mathbb{Q}$HS or equivalently a $\mathbb{Q}$-homology solid torus. To study $\widetilde{G}$ representations of $M$ whose restrictions to $\pi_1(\partial M)$ are elliptic, Culler and Dunfield gave the following definition of translation extension locus.

\begin{definition}(See \cite{CD16} Section 4)
Let $PE_{\widetilde{G}}(M)$ be the subset of representations in $R_{\widetilde{G}}(M)$ whose restriction to $\pi_1(\partial M)$ are either elliptic, parabolic, or central. Consider the composition
\[
PE_{\widetilde{G}}(M)\subset R_{\widetilde{G}}(M)\stackrel{\iota^*}{\longrightarrow} R_{\widetilde{G}}(\partial M)\stackrel{\text{trans}}{\longrightarrow} H^1(\partial M; \mathbb{R})
\]
The closure in $H^1(\partial M; \mathbb{R})$ of the image of $PE_{\widetilde{G}}(M)$ under $\text{trans}\circ\iota^*$ is called translation extension locus and denoted $EL_{\widetilde{G}}(M)$.
\end{definition}

They showed that translations extension locus of a $\mathbb{Q}$-homology solid torus satisfies the following properties. 
(In the statement of the theorem, $D_{\infty}(M)$ is the infinite dihedral group $\mathbb{Z}\rtimes \mathbb{Z}/2\mathbb{Z}$. It acts on $\mathbb{R}^2$ by translating along the x-axis by integers and reflecting about the origin.)
\begin{theorem}\cite[Theorem 4.3]{CD16}\label{CDThm0}
The extension locus $EL_{\widetilde{G}}(M)$ is a locally finite union of analytic arcs and isolated points. It is invariant under $D_{\infty}(M)$ with quotient homeomorphic to a finite graph. The quotient contains finitely many points which are ideal or parabolic. The locus $EL_{\widetilde{G}}(M)$ contains the horizontal axis $L_{0}$, which comes from representations to $\widetilde{G}$ with abelian image.
\end{theorem}

The translations extension locus depicts the set of peripherally elliptic and parabolic $\widetilde{G}$ representations of a $\mathbb{Q}$-homology solid torus. The following results regarding orderability of Dehn filling were obtained using translation extension locus.

\begin{theorem}\cite[Theorem 7.1]{CD16}\label{CDThm1}
Suppose that $M$ is a longitudinally rigid (defined in Section \ref{5}) irreducible $\mathbb{Q}$-homology solid torus and that the Alexander polynomial of $M$ has a simple root $\xi$ on the unit circle. When $M$ is not a $\mathbb{Z}$-homology solid torus, further suppose that $\xi^k\neq 1$ where $k>0$ is the order of the homological longitude $\lambda$ in $H_1(M; \mathbb{Z})$. Then there exists $a>0$ such that for every rational $r\in(-a, 0)\cup(0, a)$ the Dehn filling $M(r)$ is orderable.
\end{theorem}

\begin{theorem}\cite[Theorem 1.4]{CD16}\label{CDThm2}
Let $K$ be a hyperbolic knot in a $\mathbb{Z}$-homology $3$-sphere $Y$ . If the trace field of the knot exterior M has a real embedding then:
\begin{itemize}
\item[(a)] For all sufficiently large $n$, the $n$-fold cyclic cover of $Y$ branched over $K$ is orderable.
\item[(b)] There is an interval I of the form $(-\infty,a)$ or $(a,\infty)$ so that the Dehn filling $M(r)$ is orderable for all rational $r\in I$ .
\item[(c)] There exists $b>0$ so that for every rational $r\in (-b,0)\cup(0,b)$ the Dehn filling $M(r)$ is orderable.
\end{itemize}
\end{theorem}

Recently, Herald and Zhang \cite{HZ} improved Theorem \ref{CDThm1} in the case of $M$ being a $\mathbb{Z}$-homology solid torus by removing the longitudinally rigid condition (meaning that $M(0)$ has 0-dimensional PSL$_2\mathbb{C}$ character variety apart from the component of reducible representations) of $M$. Their result is stated as follows.
\begin{theorem}
Let $M$ be the exterior of a knot in an integral homology $3$-sphere such that $M$ is irreducible. If the Alexander polynomial $\Delta(t)$ of $M$ has a simple root on the unit circle, then there exists a real number $a>0$ such that, for every rational slope $r\in(-a, 0)\cup(0, a)$, the Dehn filling $M(r)$ has left-orderable fundamental group.
\end{theorem}

\subsection{Holonomy extension locus}
In Section \ref{loci} of this paper, I will construct holonomy extension locus which depicts the set of peripherally hyperbolic and parabolic $\widetilde{G}$ representations of a $\mathbb{Q}$-homology solid torus, in contrast to translation extension locus. 

Let $M$ be the complement of a knot in a $\mathbb{Q}$HS or equivalently a $\mathbb{Q}$-homology solid torus. We define the holonomy extension locus of $M$ as follows.

\begin{customDef}{\ref{HEL}}
Let $PH_{\widetilde{G}}(M)$ be the subset of representations whose restriction to $\pi_1(\partial M)$ are either hyperbolic, parabolic, or central. Consider the composition
\begin{displaymath}
PH_{\widetilde{G}}(M)\subset R^{\text{aug}}_{\widetilde{G}}(M)\stackrel{\iota^*}{\longrightarrow} R^{\text{aug}}_{\widetilde{G}}(\partial M)\stackrel{\text{EV}}{\longrightarrow} H^1(\partial M; \mathbb{R})\times H^1(\partial M; \mathbb{Z})
\end{displaymath}
The closure of $\text{EV}\circ \iota^* (PH_{\widetilde{G}}(M))$ in $H^1(\partial M; \mathbb{R})$ is called the holonomy extension locus and denoted $HL_{\widetilde{G}}(M)$.
\end{customDef}

Despite the similarity between the definition of translation extension locus and the holonomy extension locus, there are more technical difficulties to deal with in the peripherally hyperbolic case than the peripherally elliptic case. For instance, the translation extension locus of a $\mathbb{Q}$-homology solid torus is a planer graph in $\mathbb{R}^2$, while the holonomy extension locus has infinitely many sheets each of which is a planer graph.
This major difference comes from the fact that, for peripherally elliptic $\widetilde{G}$ representations, the translation number captures sufficient information we need, while in the hyperbolic case, translation number is far less than enough and other invariants need to be taken into consideration.
The following theorem describes the structure of a holonomy extension locus.

Suppose $\lambda$ is the homological longitude of $M$, with its order in $H_1(M; \mathbb{Z})$ being $n$. 
Define $k_M=\min \{ -\chi(S) | S\subset M \text{ a connected incompressible surface bounding } n\lambda\}$. Let $i,j$ be integers.

\begin{customthm}{\ref{graph}}
The holonomy extension locus $HL_{\widetilde{G}}(M)=\bigsqcup_{i,j\in\mathbb{Z}}H_{i,j}(M)$, $-\frac{k_M}{n}\leq j\leq \frac{k_M}{n}$ is a locally finite union of analytic arcs and isolated points. It is invariant under the affine group $D_{\infty}(M)$ with quotient homeomorphic to a finite graph with finitely many points removed. Each component $H_{i, j}(M)$ contains at most one parabolic point and has finitely many ideal points locally.

The locus $H_{0,0}(M)$ contains the horizontal axis $L_0$, which comes from representations to $\widetilde{G}$ with abelian image.
\end{customthm}

In Section \ref{examples}, plots of the holonomy extension loci of several $\mathbb{Q}$-homology solid tori will be shown. The relation between left-orderability of Dehn filling and holonomy extension locus will be demonstrated. From these examples, we will see that the holonomy extension locus provides different information about left-orderability from translation extension locus. This motivates us to obtain two theorems in Section \ref{5}.

\subsection{Main result of this paper}
Using holonomy extension loci, I study $\mathbb{Q}$HSs coming from Dehn filling on $\mathbb{Q}$-homology solid torus and construct intervals of left-orderable Dehn fillings. The following are the two main applications of the results of this paper.

\begin{customthm}{\ref{Thm_Alex}}
Suppose $M$ is the exterior of a knot in a $\mathbb{Q}$-homology $3$-sphere that is longitudinal rigid. If the Alexander polynomial $\Delta_M$ of $M$ has a simple positive real root $\xi\neq 1$, then there exists a nonempty interval $(-a, 0]$ or $[0, a)$ such that for every rational $r$ in the interval, the Dehn filling $M(r)$ is orderable. 
\end{customthm}

\begin{customthm}{\ref{Galois}}
Suppose $M$ is a hyperbolic $\mathbb{Z}$-homology solid torus. Assume the longitudinal filling $M(0)$ is a hyperbolic mapping torus of a homeomorphism of a genus $2$ orientable surface and its holonomy representation has a trace field with a real embedding at which the associated quaternion algebra splits. Then every Dehn filling $M(r)$ with rational $r$ in an interval $(-a, 0]$ or $[0, a)$ is orderable.
\end{customthm}

\subsection*{Acknowledgements}
The author was supported by NSF grant DMS-1510204 and DMS-1811156 of the United States, and NRF Mid-Career Researcher Program  (Grant No. 2018R1A2B6004003) of the Republic of Korea.
The author gratefully thanks her advisor Nathan Dunfield for introducing her to this problem and providing numerous helpful advice, and thank Steven Boyer and Ying Hu for their advice.
The author would also like to thank Marc Culler and Nathan Dunfield for allowing her to use their program which produces many of the graphs of holonomy extension loci in this paper and teaching her how it works.
Finally, the author would like to thank the referee and the editor for pointing out the mistakes and giving all the detailed advice to make this paper more readable.

\section{Background}
In the L-space conjecture \cite{BGW}, we study $\mathbb{Q}$-homology (and also $\mathbb{Z}$-homology) $3$-spheres, abbreviated as $\mathbb{Q}$HS ($\mathbb{Z}$HS). They are Dehn fillings on $\mathbb{Q}$-homology ($\mathbb{Z}$-homology) solid tori, where a $\mathbb{Q}$-homology ($\mathbb{Z}$-homology) solid torus is a compact $3$-manifold with a torus boundary whose rational (integral) homology groups are the same as those of a solid torus.

\subsection{Preliminaries in graph theory}
To study holonomy extension locus, we need some basic definitions from graph theory.
We call a graph finite if both of its edge set and vertex set are finite. In fact, a holonomy extension locus could be viewed as the union of infinite sheets, each of which is a planer graph but still slightly different from a finite graph. It contains both a finite graph part and finitely many branches going to infinity. So we need some proper notion to describe it, and we can use the notion finite graph with finitely many points removed.

\subsection{Representation Variety and Character Variety}\label{variety}
To study representations of the fundamental group, we need some tools from algebraic geometry.

An affine algebraic set is defined to be the zeros of a set of polynomials. In this paper, we also need real semialgebraic sets \cite[Chapter 3]{realAG}, which are defined by polynomial inequalities.
The dimension of a real semialgebraic set is equal to its topological dimension.
An affine algebraic variety is an irreducible affine algebraic set.

With these notions, we can define the representation variety and character variety of a $3$-manifold $M$.
We are interested in representations into Lie groups PSL$_2\mathbb{C} \simeq \text{PGL}_2\mathbb{C}$ and PSL$_2\mathbb{R}$.
The set of PSL$_2\mathbb{C}$ representations, $\text{Hom}(\pi_1(M),\text{PSL}_2\mathbb{C})$ is an affine algebraic set in some $\mathbb{C}^n$ equipped with Zariski topology. We call it the PSL$_2\mathbb{C}$ representation variety of $M$ and denote it by $R(M)$. Similarly we can define $R(\partial M)=\text{Hom}(\pi_1(\partial M),\text{PSL}_2\mathbb{C})$. The group PSL$_2\mathbb{C}$ acts on $R(M)$ by conjugation, so we can consider the geometric invariant theory (GIT) quotient $R(M)//\text{PSL}_2\mathbb{C}$, which we denote by $X(M)$. It is called the PSL$_2\mathbb{C}$ character variety of $M$.

Recall $G=\text{PSL}_2\mathbb{R}$, $\widetilde{G}=\widetilde{\text{PSL}_2\mathbb{R}}$.
Similarly we can consider the $G$ representation variety $R_G(M)$ (and $R_G(\partial M)$) and $\widetilde{G}$ representation variety $R_{\widetilde{G}}(M)$ (and $R_{\widetilde{G}}(\partial M)$). Also we define the $G$ character variety $X_G(M)$ to be the geometric invariant theory quotient $R_G(M)//\text{PGL}_2\mathbb{R}$. Both $R_G(M)$ and $X_G(M)$ are real algebraic varieties.

A rational map $f: X \rightarrow Y$ between two irreducible varieties is called dominant if $f(X)$ contains a non-empty open subset in $Y$. It is called birational if it is dominant, and if there is another dominant rational map $g: Y \rightarrow X$ such that $g\circ f = \text{id}_X$ and $f\circ g =\text{id}_Y$ as rational maps. The readers could refer to standard textbooks on algebraic geometry like \cite[Chapter I, Section 4]{Hartshorne}, for definitions of other related terminologies.
Let $f: \widehat{X}(M)\rightarrow X(M)$ be a birational map with $\widehat{X}(M)$ a smooth projective curve. Then $\widehat{X}(M)$ is called the smooth projectivization of $X(M)$. Points in $\widehat{X}(M)-f^{-1}(X(M))$ are called ideal points.
To each ideal point, we can associate incompressible surfaces to it. See \cite{CCGLS} for more details.

\subsection{$\widetilde{\text{PSL}_2\mathbb{R}}$}\label{psl2rtilde}
Consider the Lie group $SU(1, 1) = \left\{
\begin{pmatrix}
\alpha & \beta\\
\overline{\beta} & \overline{\alpha} \\
\end{pmatrix}
|\ |\alpha|^2 - |\beta|^2 = 1 \right\}$, which manifests as the isometry group of the hyperbolic plane in the disc model. So there is an isomorphism between SU$(1,1)$ and SL$_2\mathbb{R}$, which is the isometry group in the upper half-plane model.
We can parameterize $SU(1, 1)$ by $(\gamma, \omega)$, with $\gamma = -\overline{\beta}/\alpha\in \mathbb{C}$ and $\omega = \text{arg} \alpha$ defined modulo $2\pi$.
Consequentially SL$_2\mathbb{R}$ can be described as $\{(\gamma, \omega)\ |\ |\gamma|<1, -\pi \leq \omega<\pi\}$. As the universal cover of SL$_2\mathbb{R}$ and $G=\text{PSL}_2\mathbb{R}$, $\widetilde{G}=\widetilde{\text{PSL}_2\mathbb{R}}$ is also a Lie group and can be described as $\{(\gamma, \omega)\in \mathbb{C}\times \mathbb{R}\ |\ |\gamma|<1, -\infty<\omega<\infty\}$ with group operation given by:

\begin{equation}\label{multiplication}
\begin{split}
&(\gamma, \omega)(\gamma', \omega')= \\ &\left((\gamma+\gamma'e^{-2i\omega})(1+\bar{\gamma}\gamma'e^{-2i\omega})^{-1}, \omega+\omega'+\frac{1}{2i}\log{(1+\bar{\gamma}\gamma'e^{-2i\omega})(1+\gamma\bar{\gamma}'e^{2i\omega})^{-1}}
\right).
\end{split}
\end{equation}
This multiplication formula could be checked using the correspondence between $(\gamma, \omega)$ and
$\begin{pmatrix}
\alpha & \beta\\
\overline{\beta} & \overline{\alpha} \\
\end{pmatrix}$. 
So we have a copy of $\mathbb{R}$ sitting inside $\widetilde{G}$ as an abelian subgroup.

The following properties of $\widetilde{G}$ can be found in \cite[Section 2]{Khoi}. 
The universal cover of $S^1$ is $\mathbb{R}$, where $S^1$ can be viewed as lifted to unit length intervals. Being the universal cover of $G$ which acts on $S^1=P^1(\mathbb{R})$ by M\"{o}bius transformation, $\widetilde{G}$ acts on $\mathbb{R}$ faithfully so it is left-orderable. For elements in $\widetilde{G}$, define the translation number to be
\begin{displaymath}
 \text{trans}(\widetilde{g})=\lim_{n\to\infty}\frac{\widetilde{g}^n(x)-x}{n} \text{ for some } x\in \mathbb{R}.
\end{displaymath}
It's independent of the choice of $x$. The proof of the existence of the limit and more properties of the translation number could be found in \cite[Section 5]{Ghys}. The translation number is also called rotation number in some other occasions.

Let $A\in \text{SL}_2\mathbb{R}$, $A\neq \pm Id$. Then
$A$ is called elliptic if $|\text{trace}(A)|<2$ and in this case $A$ is conjugate to a matrix of the form
\begin{displaymath}
\left[\begin{matrix}\cos(\alpha) & \sin(\alpha)\\ -\sin(\alpha)&\cos(\alpha)\end{matrix}\right],  0\leq \alpha < 2\pi.
\end{displaymath}
The matrix $A$ is called parabolic if $|\text{trace}(A)|=2$ and it is conjugate to a matrix of the form
\begin{displaymath}
\pm\left[\begin{matrix}1 & 2u\\ 0 & 1\end{matrix}\right], -\infty < u < \infty.
\end{displaymath}
The matrix $A$ is called hyperbolic if $|\text{trace}(A)|>2$ and in this case it is conjugate to a matrix of the form
\begin{displaymath}
\left[\begin{matrix} a & 0\\ 0 & a^{-1}\end{matrix}\right], a\neq 0.
\end{displaymath}

Elements of SU$(1, 1)$ are classified in the same way via the identification $\text{SU}(1, 1)\simeq \text{SL}_2\mathbb{R}$. We then call an element of $\widetilde{G}$ elliptic, parabolic or hyperbolic if it covers an element of the corresponding type in SU$(1, 1)$.
By Lemma 2.1 in \cite{Khoi}, conjugacy classes in $\widetilde{G}$ can be presented as

\begin{itemize}
\item elliptic: $(0,\alpha)$, with $-\infty< \alpha/2\pi <\infty$ the translation number of elements in the conjugacy class.

\item parabolic: $\displaystyle(\frac{iu}{1+iu}, \tan^{-1}(u)+2k\pi)$ or $\displaystyle(\frac{iu}{1+iu}, \tan^{-1}(u)+\pi+2k\pi)$ , with $u\in \mathbb{R}$ and $k\in \mathbb{Z}$ the translation number of elements in the conjugacy class.

\item hyperbolic: $\displaystyle(\frac{a-a^{-1}}{a+a^{-1}}, 2k\pi)$ with $a>0$ or $\displaystyle(\frac{a-a^{-1}}{a+a^{-1}}, \pi+2k\pi)$  with $a<0$. And $k\in\mathbb{Z}$ is the translation number in both cases.
\end{itemize}
In particular, if $\widetilde{g}$ is conjugate to $(0, 2k\pi)$ or $(0, (2k+1)\pi)$, then $\widetilde{g}$ is called central, with $k\in \mathbb{Z}$ the translation number.

\subsection{Augmented Representation Variety and Character Variety}
As a subgroup of PSL$_2\mathbb{C}$, $G$ acts on $P^1(\mathbb{C})$ by the M\"{o}bius transformation as well as on $S^1=P^1(\mathbb{R})\subset P^1(\mathbb{C})$.
An element in PSL$_2\mathbb{C}$ has at least one fixed point when acting on $P^1(\mathbb{C})$. When there is more than one fixed point, we sometimes need to specify which one we are using.
So we consider the augmented representation variety and augmented character variety. More details could be found in \cite[Section 10]{augment}.

A subgroup of $G$ may not have a common fixed point on $P^1(\mathbb{C})$, but an abelian subgroup does. In fact, a nontrivial abelian subgroup of $G$ contains only one type of elliptic, hyperbolic or parabolic elements, and consequently has either one (if the subgroup contains parabolic elements) or two common fixed points(if the subgroup contains hyperbolic or elliptic elements) on $P^1(\mathbb{C})$.

Let $R^{\text{aug}}_{G}(M)$ be the subvariety of $R_{G}(M) \times P^1(\mathbb{C})$ consisting of pairs $(\rho, z)$, where $z$ is a fixed point of $\rho(\pi_1(\partial M))$. Let $X^{\text{aug}}_{G}(M)$ be the GIT quotient of $R^{\text{aug}}_{G}(M)$ under the diagonal action of $G$ by conjugation and M\"{o}bius transformations. There is a natural regular map $\pi: X^{\text{aug}}_{G}(M)\rightarrow X_{G}(M)$ which forgets the second factor.

Notice that a fixed point of a matrix in $G$ can also be viewed as its eigenvector. Eigenvalues of images of peripherally hyperbolic and parabolic representations play an important part in the definition of the holonomy extension locus in Section \ref{loci}. We need the augmented character variety $X^{\text{aug}}_G(M)$ so that given $\gamma \in \pi_1(\partial M)$, we can define a regular function $e_{\gamma}$ which sends $[(\rho, z)]$ to the square of the eigenvalue of $\rho(\gamma)$ corresponding to $z$. In contrast, on $X_G(M)$  such functions could not be defined because the two eigenvalues of $\rho(\gamma)$ cannot be distinguished. 

The fiber of $\pi: X^{\text{aug}}_{G}(M)\rightarrow X_{G}(M)$ contains two points except at $[\rho]$ where $\rho|_{\pi_1(\partial M)}$ is parabolic (fiber has one point), or $\rho|_{\pi_1(\partial M)}$ is trivial (fiber is isomorphic to $P^1(\mathbb{C})$).

\subsection{Augmented $\widetilde{\text{PSL}_2\mathbb{R}}$  Representations}
Similarly, we construct augmented $\widetilde{\text{PSL}_2\mathbb{R}}$ representations.

As a subgroup of PSL$_2\mathbb{C}$, $G$ acts on $P^1(\mathbb{C})$. There is a natural action of $\widetilde{G}$ on $P^1(\mathbb{C})$ by projecting to $G$. So hyperbolic and elliptic elements of $\widetilde{G}$ have two fixed points while parabolic elements have one fixed point on $P^1(\mathbb{C})$.
An abelian subgroup of $\widetilde{G}$ has at least one fixed point on $P^1(\mathbb{C})$. So $\rho(\pi_1(\partial M))$ has at least one fixed point.

Consider the following subset of $\widetilde{G}\times P^1(\mathbb{C})$,
\begin{displaymath}
\text{Aug}(\widetilde{G})=\{(\widetilde{A}, v)|\widetilde{A} \in \widetilde{G}, v\in P^1(\mathbb{C}) \text{ is a fixed point of } \widetilde{A}\}.
\end{displaymath}
Denote by $A\in G$ the projection of $\widetilde{A}\in \widetilde{G}$. Notice that $v$ is in fact a fixed point of $A$ on $P^1(\mathbb{C})$.
Then for any element $(\widetilde{A}, v)$ in Aug$(\widetilde{G})$ with $\widetilde{A}$ hyperbolic, we can use $(\frac{a-a^{-1}}{a+a^{-1}}, k\pi)$ as the representative of the conjugacy class of $\widetilde{A}$ in $\widetilde{G}$, where $a$ is the eigenvalue of $A$ corresponding to $v$. The sign of $a$ doesn't matter since $\frac{a-a^{-1}}{a+a^{-1}}$ is an even function.

We can now construct the augmented $\widetilde{G}$ representation variety of $M$.
Let $R^{\text{aug}}_{\widetilde{G}}(M)$ be the subvariety of $R_{\widetilde{G}}(M) \times P^1(\mathbb{C})$ consisting of pairs $(\widetilde{\rho}, z)$ with $z$ a fixed point of $\widetilde{\rho}(\pi_1(\partial M))$.
Similarly, define the augmented $\widetilde{G}$ representation variety of $\partial M$. Let $R^{\text{aug}}_{\widetilde{G}}(\partial M)$ be the subvariety of $R_{\widetilde{G}}(\partial M) \times P^1(\mathbb{C})$ consisting of pairs $(\widetilde{\rho}, z)$ with $z$ a fixed point of $\widetilde{\rho}(\pi_1(\partial M))$.

There is a natural projection from $R^{\text{aug}}_{\widetilde{G}}(-)$ to $R_{\widetilde{G}}(-)$ forgetting the second factor.
We call a representation in $R_{\widetilde{G}}(\partial M)$ elliptic, hyperbolic or parabolic if its image in $\widetilde{G}$ contains an element of the corresponding type, and call it central if its image contains only central elements.
We call a representation in $R^{\text{aug}}_{\widetilde{G}}(\partial M)$ elliptic, hyperbolic, parabolic or central if its projection to $R_{\widetilde{G}}(\partial M)$ is of the corresponding type.

\section{Holonomy extension locus}\label{loci}
In this section, we define the holonomy extension locus, show its structure and explain how it works.

\subsection{Definition and Main Property}

\begin{definition}
For hyperbolic element $\widetilde{g}\in \widetilde{G}$, take $v\in P^1(\mathbb{C})$ to be a fixed point of $\widetilde{g}$. Define $ev: \text{Aug}(\widetilde{G})\longrightarrow \mathbb{R}\times\mathbb{Z}$, $(\widetilde{g}, v)\mapsto(\ln(|a|), \text{trans}(\widetilde{g}))$, where $a$ is the eigenvalue of $g$ (the projection of $\widetilde{g}$ in $G$) corresponding to the eigenvector $v$.

For parabolic elements, define $ev: \text{Aug}(\widetilde{G})\longrightarrow \mathbb{R}\times\mathbb{Z}$, taking $\widetilde{g}$ to $(0, \text{trans}(\widetilde{g}))$.
\end{definition}

\begin{lemma}\label{eig}
The map ev$(-, v)$ is a group homomorphism when restricted to hyperbolic or parabolic abelian subgroup of $\widetilde{G}$, with $v\in P^1(\mathbb{C})$ any fixed point of the subgroup. As a consequence, ev$((\widetilde{\rho}(-), v)): \pi_1(\partial M)\rightarrow \mathbb{R}\times \mathbb{Z}$ is a group homomorphism for $\widetilde{\rho}$ hyperbolic or parabolic, where $v$ is a fixed point of $\widetilde{\rho}(\pi_1(\partial M))$.
\end{lemma}

\begin{proof}
Any nontrivial hyperbolic or parabolic abelian subgroup of $\widetilde{G}$ has at least one fixed point in $P^1(\mathbb{C})$ and let $v$ be any one of them.
Consider the stabilizer group Stab$(v)\subset \text{SL}_2\mathbb{R}$ of $v$.
We can define a homomorphism eig: $\text{Stab}(v)\longrightarrow \mathbb{R}^{\times}$ which takes $g\in \text{Stab}(v)$ to  $|a|$ where $gv=av$. Since $\pm I$ is the kernel, this homomorphism descends to a homomorphism from the stabilizer group of $v$ in $G$ to $\mathbb{R}^{\times}$ which we will still call eig. As trans is also a homomorphism and ev$(\widetilde{g}, v)=(\ln(\text{eig}(g)), \text{trans}(\widetilde{g}))$ for any $\widetilde{g}\in \widetilde{G}$ where $g\in G$ is the projection, it follows that ev$(-, v)$ preserves the group structure of hyperbolic or parabolic abelian subgroup of $\widetilde{G}$.

When $\widetilde{\rho}$ is hyperbolic (or parabolic), $\widetilde{\rho}(\pi_1(\partial M))$ becomes an abelian hyperbolic (or parabolic resp.) subgroup of $\widetilde{G}$, with $v$ a fixed point.
So being the composite of two homomorphisms $\widetilde{\rho}$ and ev$(-, v)$, ev$((\widetilde{\rho}(-), v)): \pi_1(\partial M)\rightarrow \mathbb{R}\times \mathbb{Z}$ is also a group homomorphism.

\end{proof}

Identifying Hom$(\pi_1(\partial M), \mathbb{R}\times\mathbb{Z})$ with $H^1(\partial M; \mathbb{R})\times H^1(\partial M; \mathbb{Z})$, we can view ev$((\widetilde{\rho}(-), v))$ as living in $H^1(\partial M; \mathbb{R})\times H^1(\partial M; \mathbb{Z})$. Let $M$ be an irreducible $\mathbb{Q}$-homology solid torus, and let $\iota: \partial M \rightarrow M$ be the inclusion map. With the above lemma, we can now define:
\begin{definition}
Let $PH_{\widetilde{G}}(M)$ be the subset of representations whose restriction to $\pi_1(\partial M)$ are either hyperbolic, parabolic, or central.
Define $EV:  R^{\text{aug}}_{\widetilde{G}}(\partial M)\longrightarrow H^1(\partial M; \mathbb{R})\times H^1(\partial M; \mathbb{Z})$ by $(\widetilde{\rho}, v)\mapsto ev((\widetilde{\rho}(-), v))$ on $\iota^*(PH_{\widetilde{G}}(M))$, where $\iota^*$ is the restriction $R^{\text{aug}}_{\widetilde{G}}(M)\longrightarrow R^{\text{aug}}_{\widetilde{G}}(\partial M)$ of representations of $\pi_1(M)$ to $\pi_1(\partial M)$.
\end{definition}

\begin{lemma}\label{inj}
Fix $v\in P^1(\mathbb{C})$. Let $H_v$ be the set of hyperbolic elements of $\widetilde{G}$ that fix $v$. Then any two elements of $H_v$ with the same image under ev$(-, v)$ are conjugate in $\widetilde{G}$.
\end{lemma}

\begin{proof}
We will use the homomorphism eig as in the proof of Lemma \ref{eig} and the property that ev$(\widetilde{g}, v)=(\ln(\text{eig}(g)), \text{trans}(\widetilde{g}))$ for any $\widetilde{g}\in \widetilde{G}$ where $g\in G$ is the projection.

Suppose $\widetilde{g}$ and $\widetilde{g}'$ are two elements in $H_v$, and $g$ and $g'$ are their projection in $G$. Then $g$ and $g'$ are conjugate if and only if they share the same set of eigenvalues. So $\widetilde{g}$ and $\widetilde{g}'$ are conjugate if and only if $\text{eig}(g)=\text{eig}(g')$ and $\text{trans}(\widetilde{g})=\text{trans}(\widetilde{g}')$, that is ev$(\widetilde{g}, v)$ = ev$(\widetilde{g}', v)$.
\end{proof}

\begin{definition}\label{HEL}
Consider the composition
\begin{displaymath}
PH_{\widetilde{G}}(M)\subset R^{\text{aug}}_{\widetilde{G}}(M)\stackrel{\iota^*}{\longrightarrow} R^{\text{aug}}_{\widetilde{G}}(\partial M)\stackrel{\text{EV}}{\longrightarrow} H^1(\partial M; \mathbb{R})\times H^1(\partial M; \mathbb{Z}).
\end{displaymath}
The closure of $\text{EV}\circ \iota^* (PH_{\widetilde{G}}(M))$ in $H^1(\partial M; \mathbb{R})\times H^1(\partial M; \mathbb{Z})$ is called the holonomy extension locus of $M$ and denoted $HL_{\widetilde{G}}(M)$.
\end{definition}
We will call a point in $HL_{\widetilde{G}}(M)$ a hyperbolic or parabolic point if it comes from a representation $\widetilde{\rho}\in PH_{\widetilde{G}}(M)$ such that $\widetilde{\rho}|_{\pi_1(\partial M)}$ is hyperbolic or parabolic. In a sense, the holonomy extension locus provides some kind of visualization of peripherally hyperbolic and parabolic $\widetilde{G}$ representations of $M$.

\begin{definition}
We call a point in $HL_{\widetilde{G}}(M)$ an ideal point if it only lies in the closure $\overline{EV\circ \iota^*(PH_{\widetilde{G}}(M))}$ but not in $EV\circ \iota^*(PH_{\widetilde{G}}(M))$, where the closure is taken in $\mathbb{R}^2 \times \mathbb{Z}^2$.
\end{definition}

\begin{lemma}\label{EV0}%[3.6]
Suppose $(\widetilde{\rho}, v)\in R^{\text{aug}}_{\widetilde{G}}(\partial M)$ is hyperbolic or central. If $EV(\widetilde{\rho}, v)(\gamma)=(0, 0)$ for some $\gamma \in \pi_1(\partial M)$, then $\widetilde{\rho}(\gamma)=1$.
\end{lemma}

\begin{proof}
It follows from Lemma \ref{inj} that $ev(\widetilde{\rho}(\gamma), v)=EV(\widetilde{\rho}, v)(\gamma)=(0, 0)$ implies $\widetilde{\rho}(\gamma)$ is conjugate to the identity element of $\widetilde{G}$. So $\widetilde{\rho}(\gamma)=1$.
\end{proof}

Suppose $\lambda$ is the homological longitude of $M$, with its order in $H_1(M; \mathbb{Z})$ being $n$.
Define
\begin{displaymath}
k_M=\min \{ -\chi(S) | S \text{ is a connected incompressible surface of $M$ that bounds } n\lambda\}.
\end{displaymath}

We will use Milnor-Wood inequality in the form of Proposition 6.5 from \cite{CD16}.
\begin{proposition}\label{bounded}
Suppose S is a compact orientable surface with one boundary component. For all $\widetilde{\rho}: \pi_1(S)\rightarrow \widetilde{G}$ one has
\begin{displaymath}
|\text{trans}(\widetilde{\rho}(\delta))|\leq max(-\chi(S), 0) \quad \text{where $\delta$ is a generator of $\pi_1(\partial S)$}.
\end{displaymath}
\end{proposition}
Applying this proposition, we see immediately that $\displaystyle \left|\text{trans}(\widetilde{\rho}(\lambda))\right|\leq \frac{k_M}{n}$.

In the next theorem, we will show that 
$$ HL_{\widetilde{G}}(M)=\bigsqcup_{i,j\in\mathbb{Z}}H_{i,j}(M), \quad -\frac{k_M}{n} \leq j \leq \frac{k_M}{n}. $$ 
Each $H_{i,j}(M):=HL_{\widetilde{G}}(M)\cap(\mathbb{R}^2\times \{i\}\times \{j\})\subset \mathbb{R}^2$ is a finite union of analytic arcs and isolated points.
Denote the infinite dihedral group $\mathbb{Z}\rtimes \mathbb{Z}/2\mathbb{Z}$ by $D_{\infty}(M)$. Then $D_{\infty}(M)$ acts on $\mathbb{R}^2\times \mathbb{Z}^2$ by translating $(x,y,i, j)$ to $(x,y,i+nk, j)$ for $i, j, k\in \mathbb{Z}$, and taking $(x,y,i, j)$ to $(-x,-y,-i,-j)$ by reflecting about $(0,0,0,0)$.
We will show that as a subset of $\mathbb{R}^2\times \mathbb{Z}^2$, $HL_{\widetilde{G}}(M)$ is invariant under the action of $D_{\infty}(M)$.

Define $L_r$ to be line of slope $-r$ going through the origin in $\mathbb{R}^2$. Then $L_0$ is the $x$-axis. 
Now we can state the theorem.

\begin{theorem}\label{graph}%[4.3]
The holonomy extension locus $HL_{\widetilde{G}}(M)=\bigsqcup_{i,j\in\mathbb{Z}}H_{i,j}(M)$, $\displaystyle -\frac{k_M}{n}\leq j\leq \frac{k_M}{n}$ is a locally finite union of analytic arcs and isolated points. It is invariant under the affine group $D_{\infty}(M)$ with quotient homeomorphic to a finite graph with finitely many points removed. Each component $H_{i, j}(M)$ contains at most one parabolic point and has finitely many ideal points locally.

The locus $H_{0,0}(M)$ contains the horizontal axis $L_0$, which comes from representations to $\widetilde{G}$ with abelian image.
\end{theorem}

\begin{remark}
If we assume the manifold $M$ is small, i.e. it has no closed essential surface, then there is no ideal point in $HL_{\widetilde{G}}(M)$. The proof is similar to \cite[Lemma 6.8]{CD16}. See Lemma \ref{small}.
\end{remark}

\begin{lemma}\label{invariant}%[6.1]
The holonomy extension locus $HL_{\widetilde{G}}(M)$ is invariant under $D_{\infty}(M)$.
\end{lemma}
\begin{proof}[Proof]% of Lemma \ref{invariant}
We will show the image $I$ of $PH_{\widetilde{G}}(M)$ under EV$\circ \iota^*$ is invariant under $D_{\infty}(M)$.
Take $(\widetilde{\rho}, v)\in PH_{\widetilde{G}}(M)$ and let $t=\text{EV}\circ\iota^*(\widetilde{\rho}, v)$ be the corresponding point in $I$. Let $s$ be the generator of the center of $\widetilde{G}$ which is isomorphic to $\mathbb{Z}$ and take any $\varphi\in H^1(M ;\mathbb{Z})$. Then $PH_{\widetilde{G}}(M) \ni \varphi\cdot\widetilde{\rho}: \gamma\mapsto\widetilde{\rho}(\gamma)s^{\varphi(\gamma)}$ is another lift of $\pi\circ\widetilde{\rho}$, where $\pi: \widetilde{G}\rightarrow G$ is the projection.
It's easy to see that $\widetilde{\rho}(\pi_1(\partial M))$ and $\varphi\cdot\widetilde{\rho}(\pi_1(\partial M))$ share the same fixed point $v$.
We can check that for any $\gamma\in\pi_1(M)$, we have ev$(\varphi\cdot\widetilde{\rho}(\gamma), v)=$ ev$(\widetilde{\rho}(\gamma)s^{\varphi(\gamma)}, v)=$ ev$(\widetilde{\rho}(\gamma), v)+(0,\varphi(\gamma))$. %where $s$ is the generator of $\mathbb{Z}=\ker(\widetilde{G})$.
So EV$\circ\iota^*(\varphi\cdot\widetilde{\rho}, v)=$ EV$\circ\iota^*(\widetilde{\rho}s^{\varphi}, v)=$ EV$\circ\iota^*(\widetilde{\rho}, v)+(0,\varphi)$. It follows that $I$ is invariant under translation by elements of $\iota^*(H^1(M; \mathbb{Z}))\subset H^1(\partial M; \mathbb{R})$.

Next, we will show $HL_{\widetilde{G}}(M)$ is invariant under reflection about the origin in $\mathbb{R}^2\times \mathbb{Z}^2$.
Define $f$ to be the element in Homeo$(\mathbb{R})$ taking $x\in \mathbb{R}$ to $-x$, and consider the conjugate action of $f$ on $\widetilde{G}$. The group $\widetilde{G}$ is preserved under this conjugation because $\pi(f\widetilde{g}f^{-1})$ has the same action as $\pi(\widetilde{g}^{-1})$ on $S^1$ for any $\widetilde{g}\in \widetilde{G}$.
Suppose $a$ is a square root of the derivative of $\pi(g)$ at $v$, then $a^{-1}$ is a square root of the derivative of $\pi(\widetilde{g}^{-1})$ at $v$ and $a^{-1}$ is a square root of the derivative of $\pi(f\widetilde{g}f^{-1})$ at $-v$.
Moreover we can check that
\begin{displaymath}
 \text{trans}(f\widetilde{g}f^{-1})=\lim_{n\to\infty}\frac{(f\widetilde{g}f^{-1})^n(0)-0}{n}= \lim_{n\to\infty}\frac{f\widetilde{g}^n(-0)-0}{n}=-\text{trans}(\widetilde{g}).
\end{displaymath}
This shows that ev$(\widetilde{\rho}(\gamma), v)=-\text{ev}(f\widetilde{\rho}f^{-1}(\gamma), -v)$ and it follows that EV$\circ\iota^*(\widetilde{\rho}, v)=-$EV$(f\widetilde{\rho}f^{-1}, -v)$. Given such an $f$, the image of $(f\widetilde{\rho}f^{-1}, -v)$ in $I$ is $-t$, proving the invariance.
\end{proof}

As a consequence of Lemma \ref{invariant}, we can now look at the quotient $PL_{\widetilde{G}}(M)=HL_{\widetilde{G}}(M)/D_{\infty}(M)$. In fact $\displaystyle PL_{\widetilde{G}}(M)=\sqcup_{-n<i<n, -k_M \leq j\leq k_M} H_{i,j}(M)/(\mathbb{Z}/2\mathbb{Z})$, where $\mathbb{Z}/2\mathbb{Z}$ acts on the disjoint union by taking $(x,y)\in  H_{i,j}(M)$ to $(-x,-y)\in  H_{-i,-j}(M)$. In particular, $\mathbb{Z}/2\mathbb{Z}$ acts on $H_{0,0}(M)$ via reflection about the origin.

\begin{lemma}\label{component}%[6.2]
$PL_{\widetilde{G}}(M)$ has finitely many connected components. In particular, each $H_{i,j}(M)$ has finitely many connected components.
\end{lemma}
\begin{proof}[Proof]%of Lemma \ref{component}
The proof works similarly as Lemma 6.2 of \cite{CD16}.

Let $\Pi: R_{\widetilde{G}}(M)\rightarrow R_{G}(M)$ be the map between representation varieties induced by $\pi: \widetilde{G}\rightarrow G$. Let $PH_{G}(M)$ be the subset of $R_{G}(M)$ consisting of representations whose restriction to $\pi_1(\partial M)$ consist only of hyperbolic, parabolic and trivial elements.
The set $PH_{G}(M)$ is a subset of the real algebraic set $R_{G}(M)$ cut out by polynomial inequalities. It follows that $PH_{G}(M)$ is a real semialgebraic set.

Let $PH^{\text{lift}}_{G}(M)\subset PH_{G}(M)$ be the image of $PH_{\widetilde{G}}(M)$ under $\Pi$. By continuity of the translation number, $PH^{\text{lift}}_{G}(M)$ is a union of connected components of $PH_{G}(M)$. Moreover $PH^{\text{lift}}_{G}(M)\subset PH_{G}(M)$ is the quotient of $PH_{\widetilde{G}}(M)$ under the action of $H^1(M, \mathbb{Z})$ and $\Pi$ is the covering map. So it is also a real semialgebraic set and thus has finitely many connected components.

The action of $H^1(M, \mathbb{Z})$ on $PH_{\widetilde{G}}(M)$ then induces an action of $\mathbb{Z}\leq D_{\infty}(M)$ on $HL_{\widetilde{G}}(M)$. Let $\Pi^{-1}(PH^{\text{lift}}_{G}(M))$ be any sheet in the covering of $PH^{\text{lift}}_{G}(M)$.
So $PL_{\widetilde{G}}(M)=\overline{EV\circ \iota^*(\Pi^{-1}(PH^{\text{lift}}_{G}(M)))}/(\mathbb{Z}/2\mathbb{Z})$, and thus has finitely many components. Let $PH^j_G(M)$ be the subset of $PH^{\text{lift}}_{G}(M)$ consisting of representations with translation number of the homological longitude being $j$. Then $PH^{\text{j}}_{G}(M)$ is a finite union of connected components of $PH_{G}(M)$. It follows that $H_{i,j}(M)=\overline{EV\circ \iota^*(\Pi^{-1}(PH^{\text{j}}_{G}(M)))}$ has finitely many components, where $\Pi^{-1}(PH^{\text{j}}_{G}(M))$ is any sheet in the covering of $PH^{\text{j}}_{G}(M)$.
\end{proof}

\begin{proof}[Proof of Theorem \ref{graph}]
First notice that the index $j$ is bounded, which follows from Proposition \ref{bounded}.

Define $c: H^1(\partial M; \mathbb{R})\times H^1(\partial M; \mathbb{Z})\rightarrow X_G(\partial M), (f_1, f_2)\mapsto \text{character of } \rho$, where $\rho$ is given by $\rho(\mu)=\left[\begin{matrix}e^{f_1(\mu)} & 0\\0 &e^{-f_1(\mu)}\end{matrix}\right],
 \rho(\lambda)=\left[\begin{matrix}e^{f_1(\lambda)} & 0\\0 &e^{-f_1(\lambda)}\end{matrix}\right]$.

Consider the dual basis $\{\mu^*, \lambda^*, m^*, l^*\}$ for $H^1(\partial M; \mathbb{R})\times H^1(\partial M; \mathbb{Z})$, where
\begin{equation} \label{dual_basis}
\mu^*(p\mu+q\lambda)=p, \lambda^*(p\mu+q\lambda)=q, m^*(p\mu+q\lambda)=p \text{ and } l^*(p\mu+q\lambda)=q
\end{equation}
for any $p\mu+q\lambda\in H_1(\partial M)$.
Take $(x,y,i,j) \in HL_{\widetilde{G}}(M) \subset H^1(\partial M; \mathbb{R})\times H^1(\partial M; \mathbb{Z})$. If we use trace-squared coordinates on $X_G(\partial M)$, we get
\begin{align*}
c(x, y, i, j)&=(\text{tr}(\rho(\mu)), \text{tr}(\rho(\lambda)), \text{tr}( \rho(\mu)\rho(\lambda) ) ) \\
               &=(e^{2x}+e^{-2x}+2, e^{2y}+e^{-2y}+2, e^{2x+2y}+e^{-2x-2y}+2).
\end{align*}
It is easy to check that $c(-x, -y, -i, -j)=c(x, y, i, j)$ and $c(x, y, i+n_1, j+n_2)=c(x, y, i, j)$, where $n_1$ and $n_2$ are integers.

Consider the diagram
\begin{displaymath}
\xymatrix{
   PH_{\widetilde{G}}(M) \ar[r]^-{\text{EV}\circ \iota^*} \ar[d] & H^1(\partial M; \mathbb{R})\times H^1(\partial M; \mathbb{Z})  \ar[d]^c\\
  X_G(M)  \ar[r]^{\iota^*}  & X_G(\partial M)  \\
  }
\end{displaymath}
The vertical map $c$ maps $HL_{\widetilde{G}}(M)$ into $\overline{\iota^*(X_G(M))}$. Being the image of a real algebraic set under a polynomial map, $X_G(M)$ is a real semialgebraic subset of $X_{\mathbb{R}}(M)$. Since $\iota^*(X(M))\subset X(\partial M)$ has complex dimension at most $1$ \cite[Lemma 2.4]{CD16}, then the real semialgebraic set $\overline{\iota^*(X_G(M))}$ has real dimension at most $1$. Moreover $\iota^*(X_G(M))$ is a locally finite graph as $X_G(M)$ is.
Thus, its preimage under $c$ is a locally finite graph invariant under $D_{\infty}(M)$ with analytic edges. So each $H_{i,j}(M)$ and thus $PL_{\widetilde{G}}(M)$ is a locally finite graph and by Lemma \ref{component} it has finitely many connected components. Therefore,  $PL_{\widetilde{G}}(M)$ is homeomorphic to a finite graph with finitely many points removed.

Suppose $D$ is a closed disc in $H^1(\partial M; \mathbb{R})$, then $D\cap H_{i,j}(M)$ lives in a finite graph. Since by Lemma \ref{component} $H_{i,j}(M)$ has finitely many components, then $D\cap H_{i,j}(M)$ also has finitely many components and thus is a finite graph. So $D\cap H_{i,j}(M)$ is the closure of a set of finitely many components in a finite graph and thus contains finitely many ideal points.

Parabolic points can only occur at the origin of each $H_{i,j}(M)$, so there can be at most one parabolic point in each component $H_{i,j}(M)$.

Recall from Section \ref{psl2rtilde} that there is an abelian subgroup of $\widetilde{G}$ that is isomorphic to $\mathbb{R}$.
Consider the diagonal representations in $G$. %=PSL_2\mathbb{R}
They lift to a one parameter family of abelian representations $\pi_1(M)\rightarrow \widetilde{G}$ by sending the generator of $H_1(M;Z)_{\text{free}}\cong H_1(M;Z)/(\text{torsion})\cong \mathbb{Z}$ to a given element in $\mathbb{R}$. Since the longitude $\lambda$ of $\partial M$ is $0$ in $H_1(M; \mathbb{Z})_{\text{free}}$, this one parameter family of abelian representations give rise to the line $L_0$ in $H_{0,0}(M)$.

\end{proof}

\subsection{Other Properties}

Recall from Section \ref{variety} that $\widehat{X}(M)$ is the smooth projectivization of $X(M)$. 
The following lemma describes some other properties of $HL_{\widetilde{G}}(M)$.
\begin{lemma}[structure of $H_{i,j}(M)$]\label{asymptote}
Suppose for some $i, j$, $H_{i,j}(M)$ contains an arc that continues on to infinity.
Then this arc approaches an asymptotes $y=-rx$ in $\mathbb{R}^2$ as it goes to infinity, where $r$ is the boundary slope of the associated incompressible surface to some ideal point of $\widehat{X}(M)$.
\end{lemma}

\begin{proof}

The vertical map $c$ in the diagram from the proof of Theorem \ref{graph} maps $HL_{\widetilde{G}}(M)$ into $\overline{\iota^*(X_G(M))}$. 
Suppose $H_{i,j}(M)$ contains an arc $A$ that continues to infinity, then there is a an ideal point $x$ of $\widehat{X}(M)$ that is the limit of a sequence of characters $\{[\rho_k]\}$ in $X(M)$ of hyperbolic representations $\{\rho_k\}$ such that images of lifts $\{\widetilde{\rho_k}\}$ under $\text{EV}\circ\iota^*$ are contained in $A$.
To show this, suppose images of $\{\widetilde{\rho_k}\}$ under $\text{EV}\circ\iota^*$ go to infinity in $HL_{\widetilde{G}}(M)$. Then $\{[\rho_k]\}$ march off to infinity in $X(M)$ as eigenvalues of either $\{\rho_k(\mu)\}$ or $\{\rho_k(\lambda)\}$ go to infinity. Thus by passing to a subsequence, $\{[\rho_k]\}$ converge to an ideal point $x$ of $\widehat{X}(M)$. Notice that traces of elliptic and parabolic elements of $G$ are bounded, by passing to a subsequence, we can assume that $\widetilde{\rho_k}|_{\pi_1(\partial M)}$ are hyperbolic. Moreover, one can choose a sequence of points $\{v_k\}$ where $v_k\in P^1(\mathbb{C})$ is a common fixed point of $\rho_k(\pi_1(\partial M))$ acting on $P^1(\mathbb{C})$. And by passing to a subsequence, we can assume $\{v_k\}$ limits to $v\in P^1(\mathbb{C})$.

By the result in \cite[Section 5.7]{CCGLS}, there exists $\beta\in \pi_1(\partial M)$ such that tr$^2_{\beta}(x)=b^2+b^{-2}+2$ is finite and $\beta=p\mu+q\lambda$, where $r=p/q$ is the boundary slope of the incompressible surface associated to the ideal point $x$.
Then $\lim_{k\to\infty} \text{tr}^2_{\beta}([\rho_k])=b^2+b^{-2}+2$ as $[\rho_k]\rightarrow x$, where $b^2$ is a positive real number as it is the limit of the square of an eigenvalue of a hyperbolic $G$ matrix. Moreover, $b$ has to be a root of unity by \cite[Section 5.7]{CCGLS}. It follows that $b^2=1$, which implies $\lim_{k\to\infty}\rho_k(\beta)=I$. It follows that $\lim_{k\to\infty}\widetilde{\rho_k}(\beta)=\widetilde{I}$, where $\widetilde{I}$ is a lift of $I$ with translation number being $\lim_{k\to\infty} \text{trans}(\widetilde{\rho_k}(\beta))=p \lim_{k\to\infty}\text{trans}(\widetilde{\rho_k}(\mu))+q\lim_{k\to\infty} \text{trans}(\widetilde{\rho_k}(\lambda))=pi+qj$.
Then we can check the slope of the asymptote of the arc containing the sequence of points $\{\text{EV}(\widetilde{\rho_k}, v_k)\}$ in $HL_{\widetilde{G}}(M)$.
It follows from direct computation that  $p\lim_{k\to\infty}\text{EV}(\widetilde{\rho_k}, v_k)(\mu)+q\lim_{k\to\infty}\text{EV}(\widetilde{\rho_k}, v_k)(\lambda)= p\lim_{k\to\infty}\text{ev}(\widetilde{\rho_k}(\mu), v_k)+q\lim_{k\to\infty}\text{ev}(\widetilde{\rho_k}(\lambda), v_k) =\lim_{k\to\infty}\text{ev}(\widetilde{\rho_k}(p\mu+q\lambda), v_k)=\text{ev}(\widetilde{I}, v)=(\ln(|b|)=0, pi+qj)$. So $\lim_{k\to\infty} \text{slope}[\rho_k]=-r$ and thus the curve $A$ is asymptotic to the line of slope $-r$ going through the origin.

\end{proof}

Holonomy extension locus can be viewed as an analog of the A-polynomial which was first introduced in \cite{CCGLS} by Cooper et al.. To explain this relation, we will start with the definition of eigenvalue variety \cite[Section 7]{Tillmann}.

Let $R_U^{\text{aug}}(M)$ be the subvariety of $R^{\text{aug}}(M)$ defined by two equations that set the lower left entries in $\rho(M)$ and $\rho(L)$ to be zero.
Consider the eigenvalue map,
\begin{displaymath}
R_U^{\text{aug}}(M)\rightarrow (\mathbb{C}-0)^2
\end{displaymath}
By taking the closure of the image of this map and discarding zero dimensional components, we get the eigenvalue variety $\mathfrak{E}(M)$ of $M$, which is defined by a principal ideal. A generator for the radical of this ideal is called the A-polynomial. We will call points that are only in the closure but not in the image ideal points.

We are only interested in the intersection of $\mathfrak{E}(M)$ with $\mathbb{R}^2$ as those points come from peripherally parabolic or hyperbolic $G$ representations. The composition $R_{G}^{\text{aug}}\rightarrow R_U^{\text{aug}}(M) \rightarrow \mathbb{R}^2\cap\mathfrak{E}(M)$ gives a map from a peripherally hyperbolic or parabolic $G$ representation $\rho$ of $M$ to eigenvalues of $\rho(\mu)$ and $\rho(\lambda)$, where $\mu$ and $\lambda$ are the meridian and longitude of $\partial M$. This map is similar to but not entirely the same as EV$\circ\iota^*$ defined in \ref{HEL}.

Recall that $M$ is called a small manifold if it contains no closed essential surface. We will prove the following lemma.
\begin{lemma}\label{small}
If $M$ is small, then there is no ideal point in $HL_{\widetilde{G}}(M)$ or $(\mathbb{R}^2-\mathbf{0})\cap\mathfrak{E}(M)$.
\end{lemma}
\begin{proof}
The proof works the same way as in \cite[Lemma 6.8]{CD16}.
Suppose $t_0$ is an ideal point in $HL_{\widetilde{G}}(M)$ (resp. $(\mathbb{R}^2-\mathbf{0})\cap\mathfrak{E}(M)$) and $\{\widetilde{\rho_i}\}\subset PH_{\widetilde{G}}(M)$ is a sequence of $\widetilde{G}$ representations whose images in $HL_{\widetilde{G}}(M)$ (resp. $(\mathbb{R}^2-\mathbf{0})\cap\mathfrak{E}(M)$) converge to $t_0$. Suppose $\{[\rho_i]\}$ is the sequence of corresponding characters in $X_G(M)$. A similar argument shows that by passing to a subsequence, $\{[\rho_i]\}$ lie in a single irreducible component $X'$ of $X(M)$ and $\{[\rho_i]\}$ either limit to a character $\chi$ in $X_G(M)$ or march off to infinity in the noncompact curve $X'$. In the latter case, as both $|\text{tr}(\rho_i(\mu))|$ and $|\text{tr}(\rho_i(\lambda)|$ are bounded above, $|\text{tr}(\rho_i(\gamma))|$ is bounded above for any $\gamma\in \pi_1(\partial M)$. The argument of \cite[Section 2.4]{CCGLS} produces a closed essential surface associated to a certain ideal point of $X'$, contradicting our hypothesis that $M$ is small.

In the case when the $[\rho_i]$ limit to $\chi$ in $X_G (M)$, a similar argument shows that $t_0$ is not actually an ideal point, proving the lemma.

\end{proof}

Finally, we use the following lemma to construct order. We focus on the $H_{0,0}(M)$ sheet of $HL_{\widetilde{G}}(M)$. Recall that $L_r$ is a line through origin in $\mathbb{R}^2$ with slope $-r$. We will see from the proof of Lemma \label{main lemma} why we require the slope to be $-r$ instead of $r$.

\begin{lemma}\label{main lemma}%[4.4]
If $L_r$ intersects $H_{0,0}(M)$ at a nonzero point that not ideal, and assume $M(r)$ is irreducible, then $M(r)$ has left-orderable fundamental group.

\end{lemma}

\begin{proof}
The idea of proof works in the following way, a point in the intersection corresponds to a $\widetilde{G}$ representation $\widetilde{\rho}$ of $\pi_1(M)$ which maps $\gamma \in \pi_1(\partial M)$ to identity, where $\gamma$ is the homology class of simple closed curves of slope $r$ on $\partial M$. Then $\widetilde{\rho}$ becomes a representation of $\pi_1(M(r))$. 

Let $f=(x_1, y_1)$ be a point in $L_r\cap H_{0,0}(M)$ that is different from the origin and not an ideal point by assumption. Then $f$ is not parabolic as parabolic points can only occur at the origin. So there exists a preimage $\widetilde{\rho}\in R_{\widetilde{G}}(M)$ of $f$ which is hyperbolic when restricting to $\pi_1(\partial M)$. Suppose $\gamma \in \pi_1(\partial M)$ realizes slope $r=j/k$, i.e. $\gamma=\lambda^k\mu^j$.
By definition of $L_r: y=-rx$, we have $f(\gamma)=\text{EV}(\widetilde{\rho})(\gamma)=\text{ev}\circ\widetilde{\rho}(\gamma)=(ky_1+jx_1, k\cdot\text{trans}(\lambda)+j\cdot\text{trans}(\mu))=(k(-jx_1/k)+jx, k0+j0)=(0,0)$. The minus sign in the slope of $L_r$ is needed so that the translation number of $\widetilde{\rho}(\gamma)$ becomes $0$. It follows from Lemma \ref{EV0} that $\widetilde{\rho}(\gamma)=1$, so we get an induced representation $\overline{\rho}: \pi_1(M(r))\rightarrow \widetilde{G}$.
As $f$ is different from the origin, then we can always find an element $\eta\in \pi_1(\partial M)$ with slope different from $r$ such that $\overline{\rho}(\eta)\neq 0$, which implies that $\overline{\rho}$ is nontrivial. So we have constructed a nontrivial  $\widetilde{G}$ representation of $ \pi_1(M(r))$.
Since $M(r)$ is irreducible by assumption, it follows from \cite[Theorem 3.2]{BRW} that $\pi_1(M(r))$ is left-orderable.

\end{proof}

\section{Examples}\label{examples}
In this section, I will show some examples of holonomy extension loci. We will see that the holonomy extension locus of a $\mathbb{Q}$-homology solid torus $M$ provides a way of visualizing the set of peripherally hyperbolic and parabolic $\widetilde{G}$ representations of $\pi_1(M)$. Moreover, we will see that together with Lemma \ref{main lemma}, the range of slopes of orderable Dehn filling on $M$ could be determined from looking at the graph of the holonomy extension locus. 

Recall from (\ref{dual_basis}) the definition of the dual basis $\{\mu^*, \lambda^*\}$.
Our first example is the figure eight knot $4_1$, whose Alexander polynomial is $t^2-3t+1$.
\begin{figure}[H]
\center
\includegraphics[width=85mm]{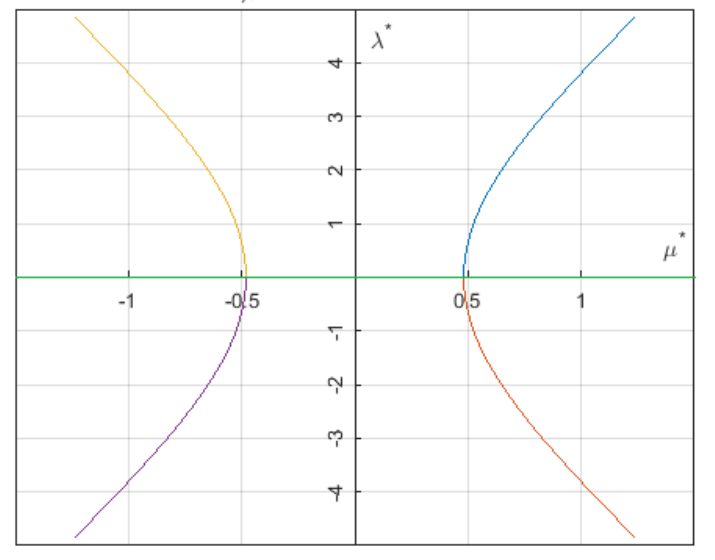}
\caption{Holonomy Extension Locus $HL_{\widetilde{G}}(4_1)$}
\label{figure8}
\fnote{There is nothing interesting going on in the translation extension locus of the figure-eight knot complement as it contains only the $x$-axis $y=0$ coming from abelian representations. The above figures shows its holonomy extension locus which has no other sheets except $H_{0,0}(M)$. 
The figure-eight knot complement has genus 1, so the $2g-1$ bound for translation number $j$ of the longitude is not sharp. 

There are two asymptotes of the graph with slopes $\pm 4$. So fillings on the figure-eight knot complement with slope lying in the interval $(-4,4)$ are orderable, by Lemma \ref{main lemma}. This observation is confirmed by Proposition 10 of \cite{BGW}.}
\end{figure}

Our next example is the $(7, 3)$ two-bridge knot $5_2$. The complement of a two-bridge knot is small \cite[Theorem 1(a)]{2bridge}. So the holonomy extension locus of the two-bridge knot does not have ideal points by Lemma \ref{small}.
\begin{figure}[H]
\centering
\includegraphics[width=60mm]{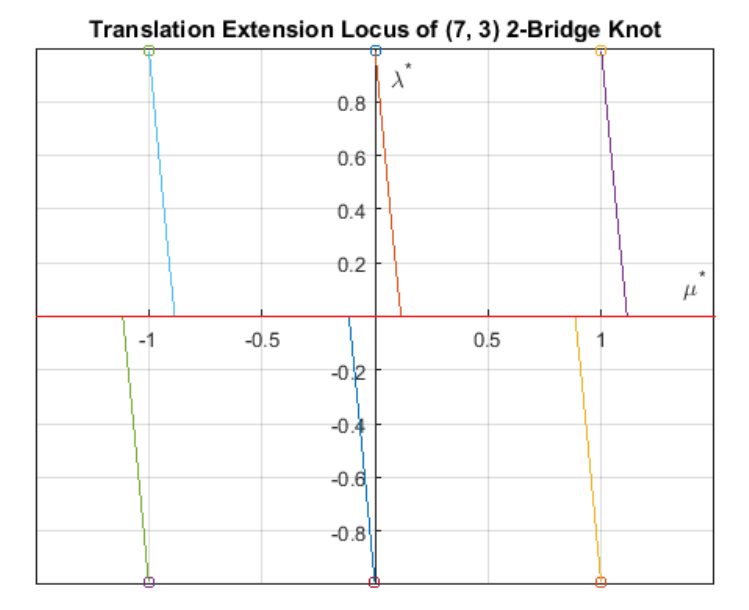}
\includegraphics[width=60mm]{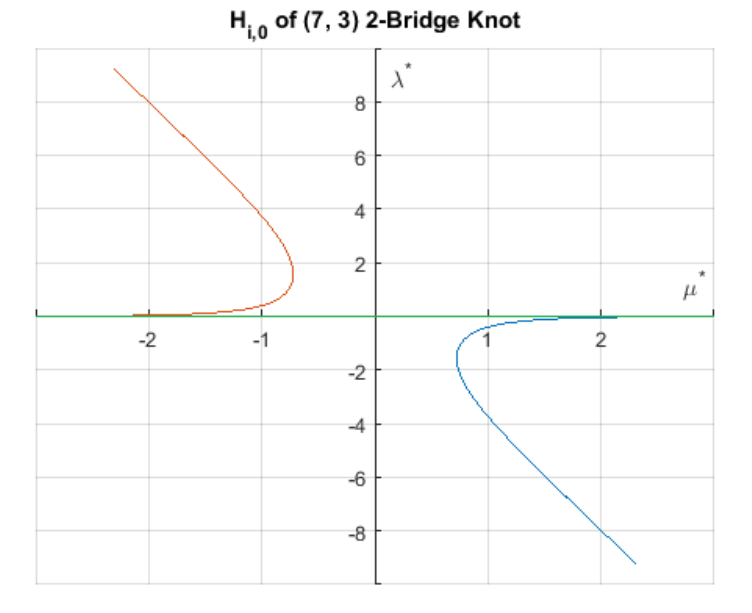}
\includegraphics[width=60mm]{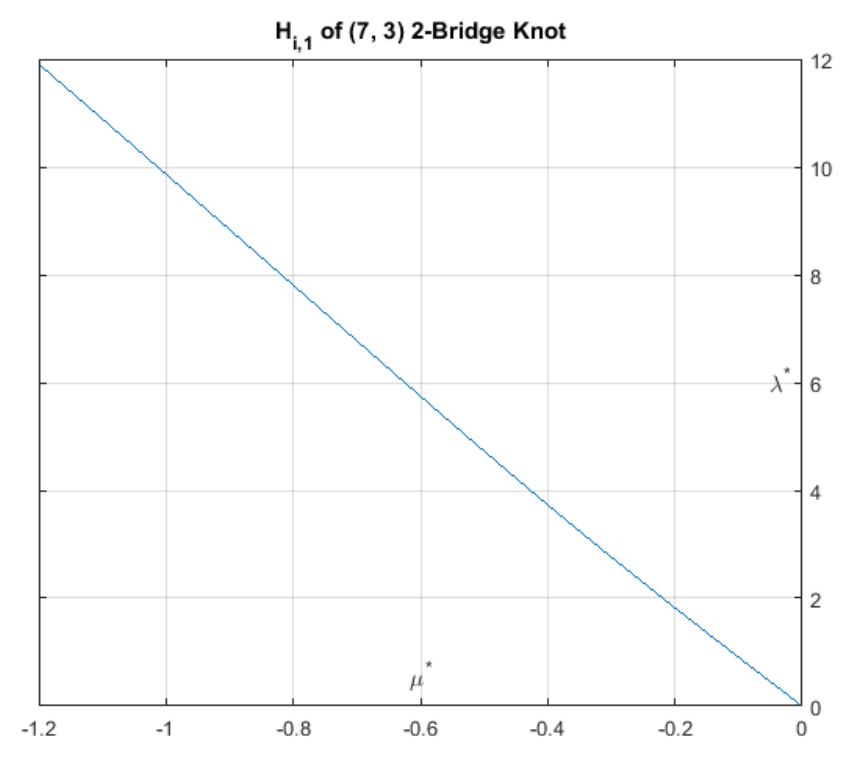}
\includegraphics[width=60mm]{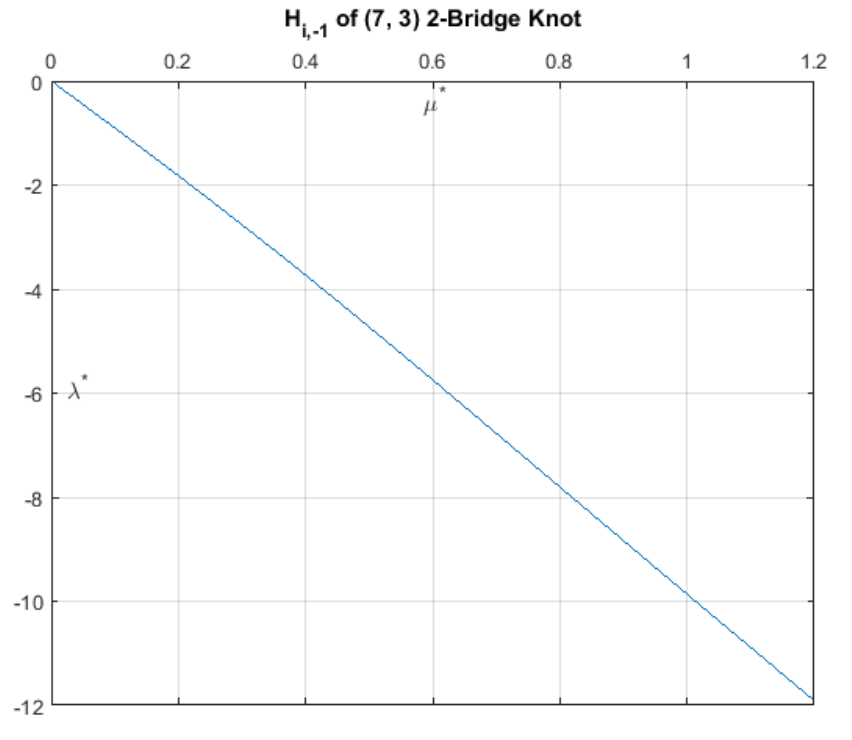}
\caption{Holonomy and Translation Extension Locus of $(7, 3)$ 2-bridge Knot}

\fnote{The top left figure is the translation extension locus of the $(7,3)$ two-bridge knot, where the six small circles are parabolic points. The translation extension locus tells us $(-\infty, 1)$ fillings are orderable, following from result of Culler-Dunfield \cite{CD16}. 

The top right figure is the $H_{0,0}(M)$ component of its holonomy extension locus. There are two asymptotes with slope $-4$ and $0$. The interval of left-orderable Dehn fillings we can read off from the holonomy extension locus is [0,4), again by Lemma \ref{main lemma}.
So compared to translation extension locus, the holonomy extension locus does tell us something more.

The two figures on the bottom are $H_{0,1}$ and $H_{0,-1}$. Notice that asymptotes in $H_{0,\pm 1}$ both have slope $-10$.
Actually, boundary slopes associated to ideal points of the character variety of the $(7, 3)$ two-bridge knot complement are $0$, $4$, $10$. This result confirms Lemma \ref{asymptote}.}
\end{figure}
The $(7,3)$ two-bridge knot, whose genus is $1$, is a twist knot of three half twist. So its Alexander polynomial is not monic and it follows that it is not fibered \cite{FiberAlexPoly}. Moreover, it cannot be an L-space knot \cite[Corollary 1.3]{fiber_lspace}.
In \cite[Section 9, Question (4)]{CD16}, it is observed that for fibered knots, the bound $2g-1$ for translation number of the longitude is never sharp. However we can see from this example that for non fibered knots, this bound can be sharp.

For the above examples, we actually computed the equations defining the curves in the graphs.
For the rest of this section, we will show some more complicated pictures produced by the program PE \cite{PE} written by Culler and Dunfield under SageMath \cite{sage}. In these examples, instead of showing the entire $HL_{\widetilde{G}}(M)$, we only show the quotient $PL_{\widetilde{G}}(M)$ of $HL_{\widetilde{G}}(M)$ under the action of $D_{\infty}(M)$, where we identify $H_{0,j}$ with $H_{0,-j}$ when $j\neq 0$, and quotient $H_{0,0}$ down by reflection about the origin.

Our first example is $t03632$, which has a loop in its holonomy extension locus.
\begin{figure}[H]
\centering
\includegraphics[width=62mm]{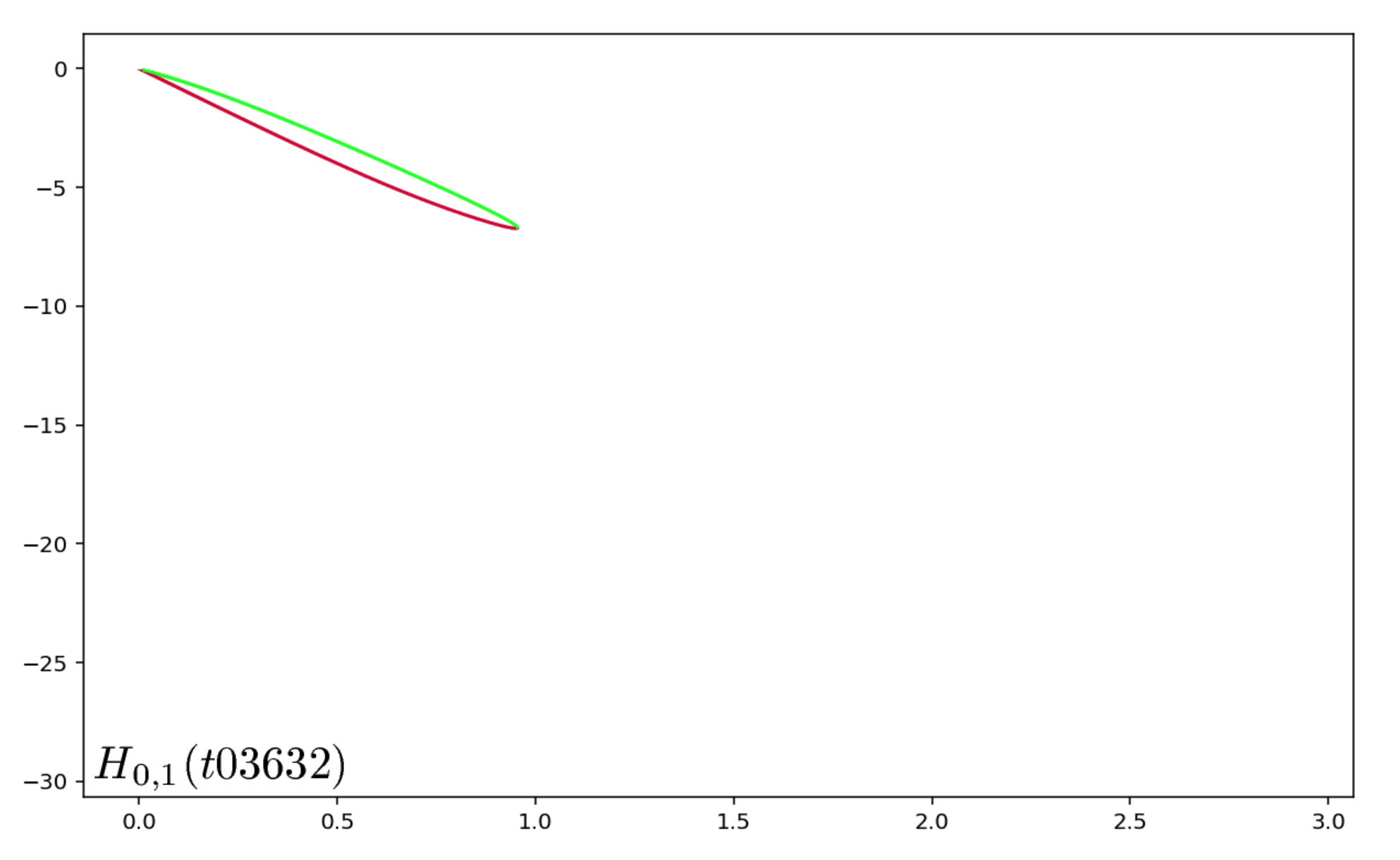}
\includegraphics[width=62mm]{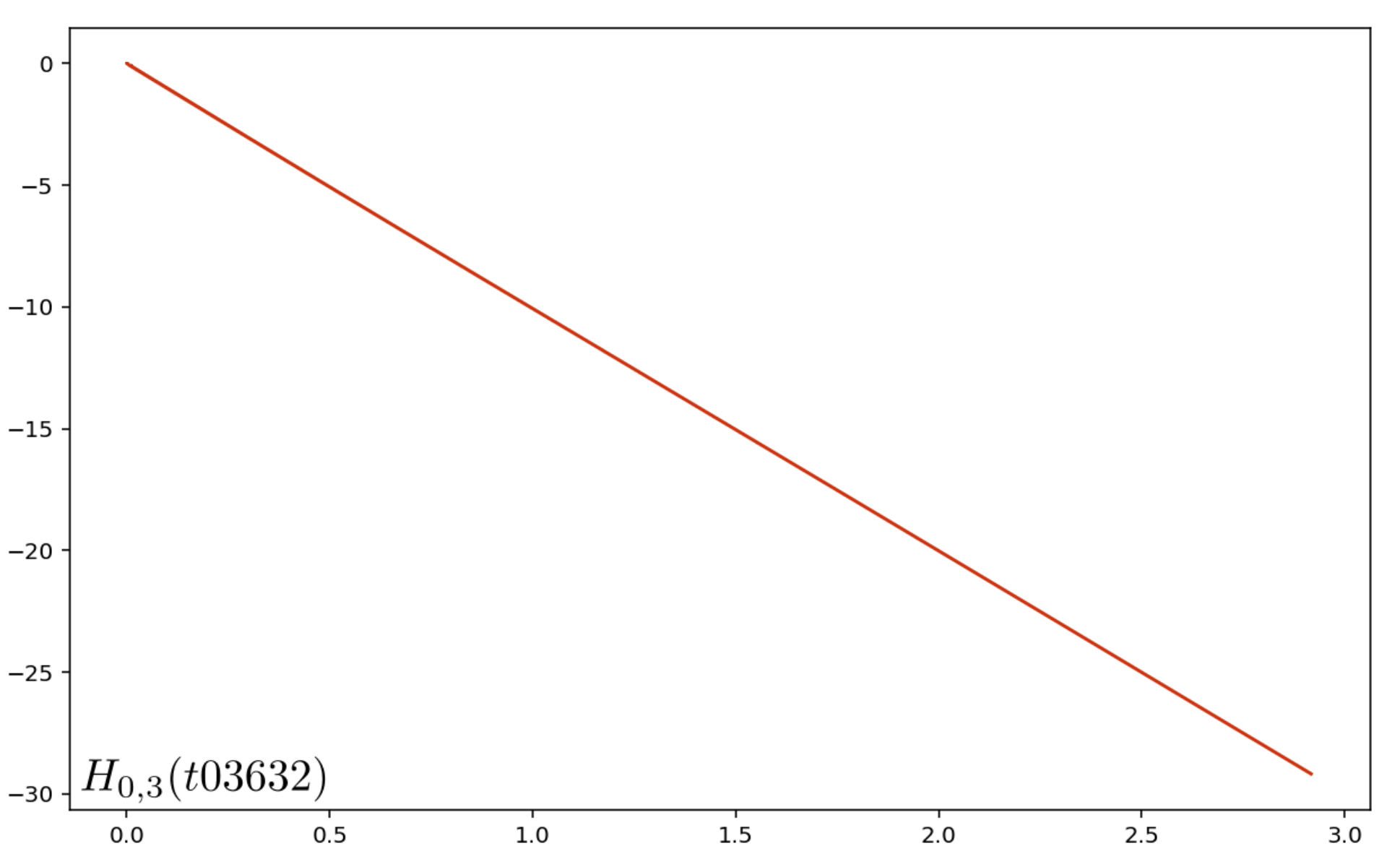}
\includegraphics[width=62mm]{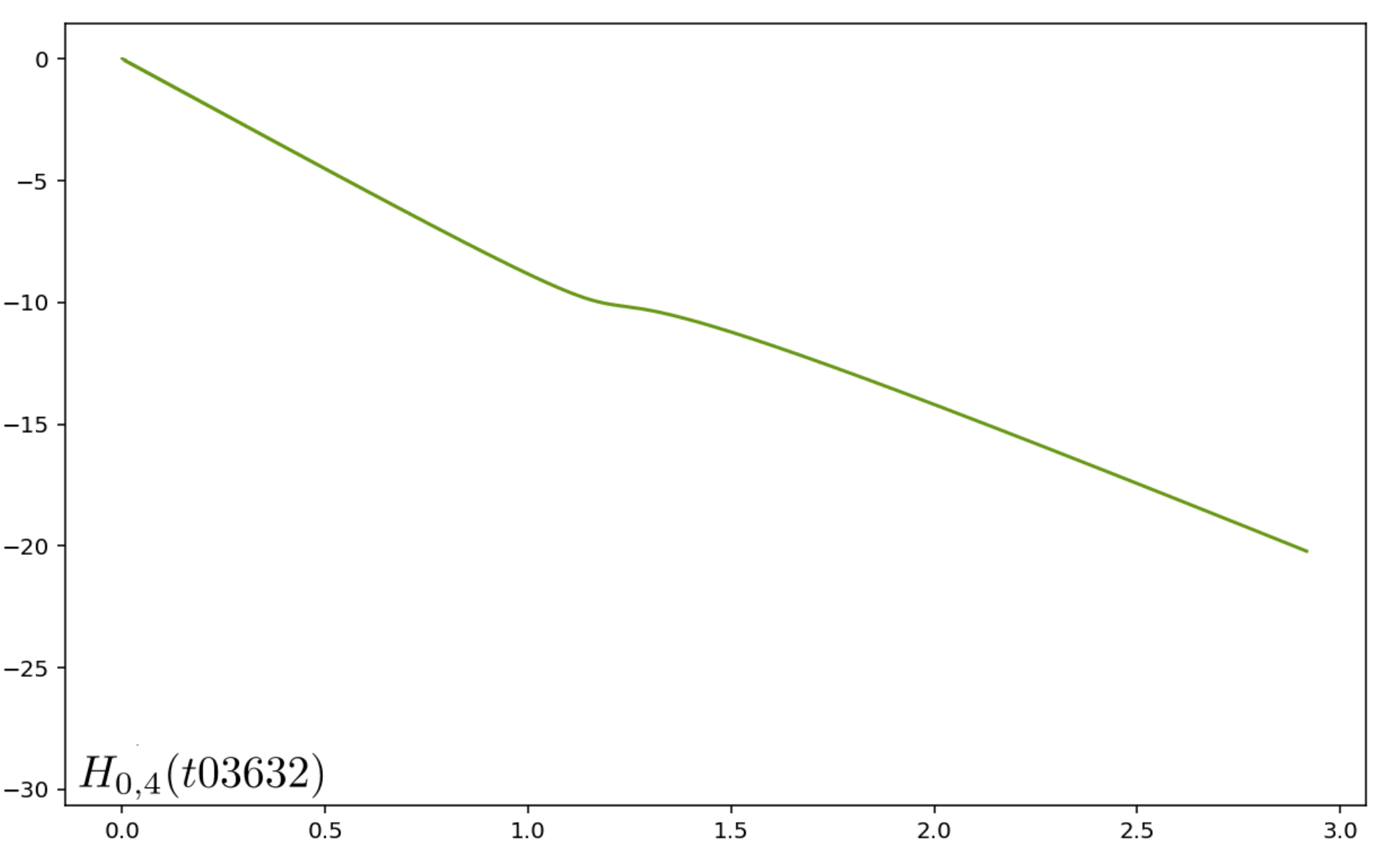}
\includegraphics[width=62mm]{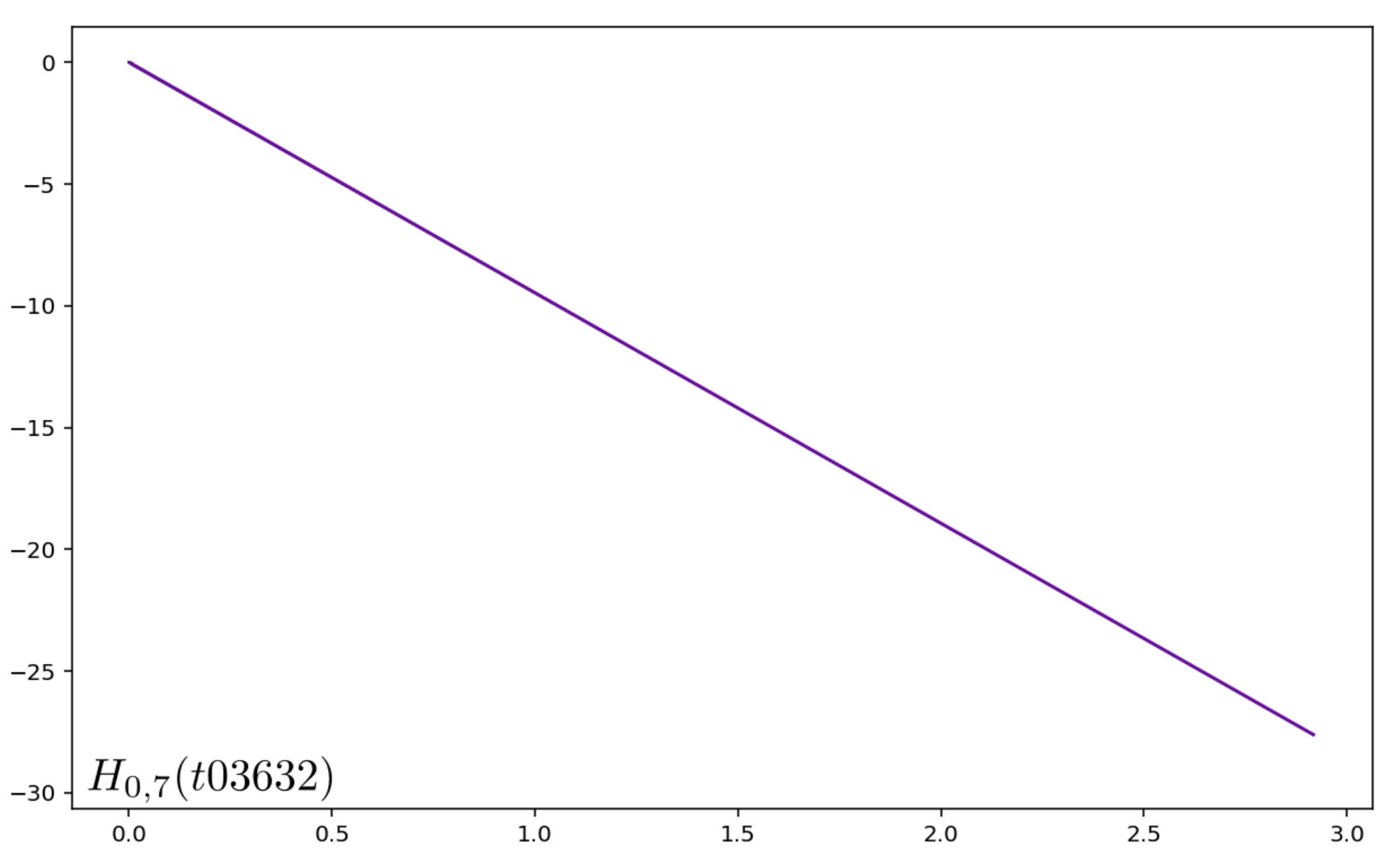}
\caption{$PL_{\widetilde{G}}(t03632)$}
\fnote{Top left figure is $H_{0,1}$ of $t03632$, where we see a small loop based at the origin (parabolic point). The Alexander polynomial of $t03632$ has no positive real root. The locus $H_{0,0}$ contains nothing other than the horizontal line representing abelian representations so we will not show it here. } % $t01741$
\end{figure}

Our next example is $7_3$ which has a more interesting $H_{0,0}$.
\begin{figure}[H]%typical example with no real root
\centering
\includegraphics[width=75mm]{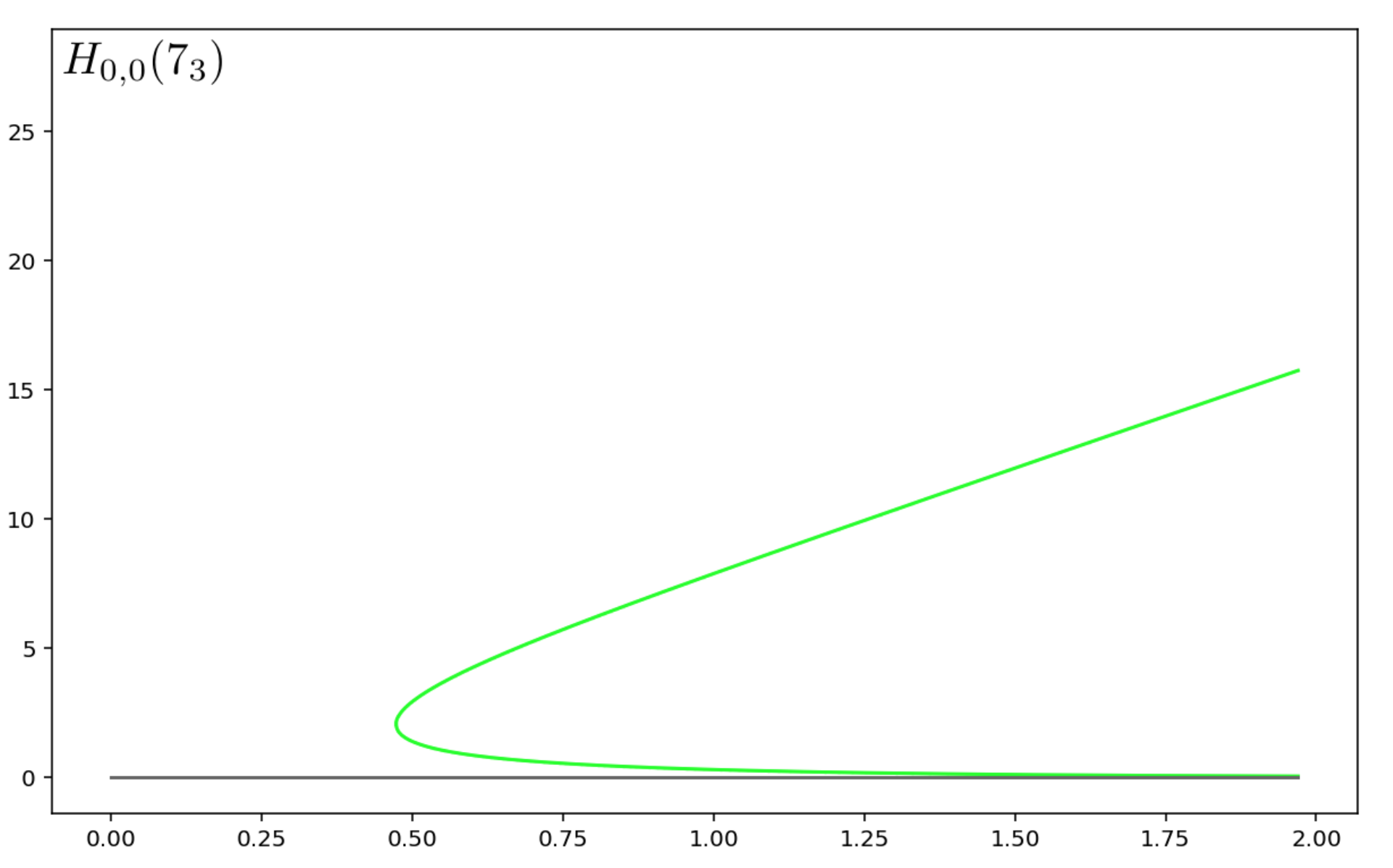}\\
\includegraphics[width=62mm]{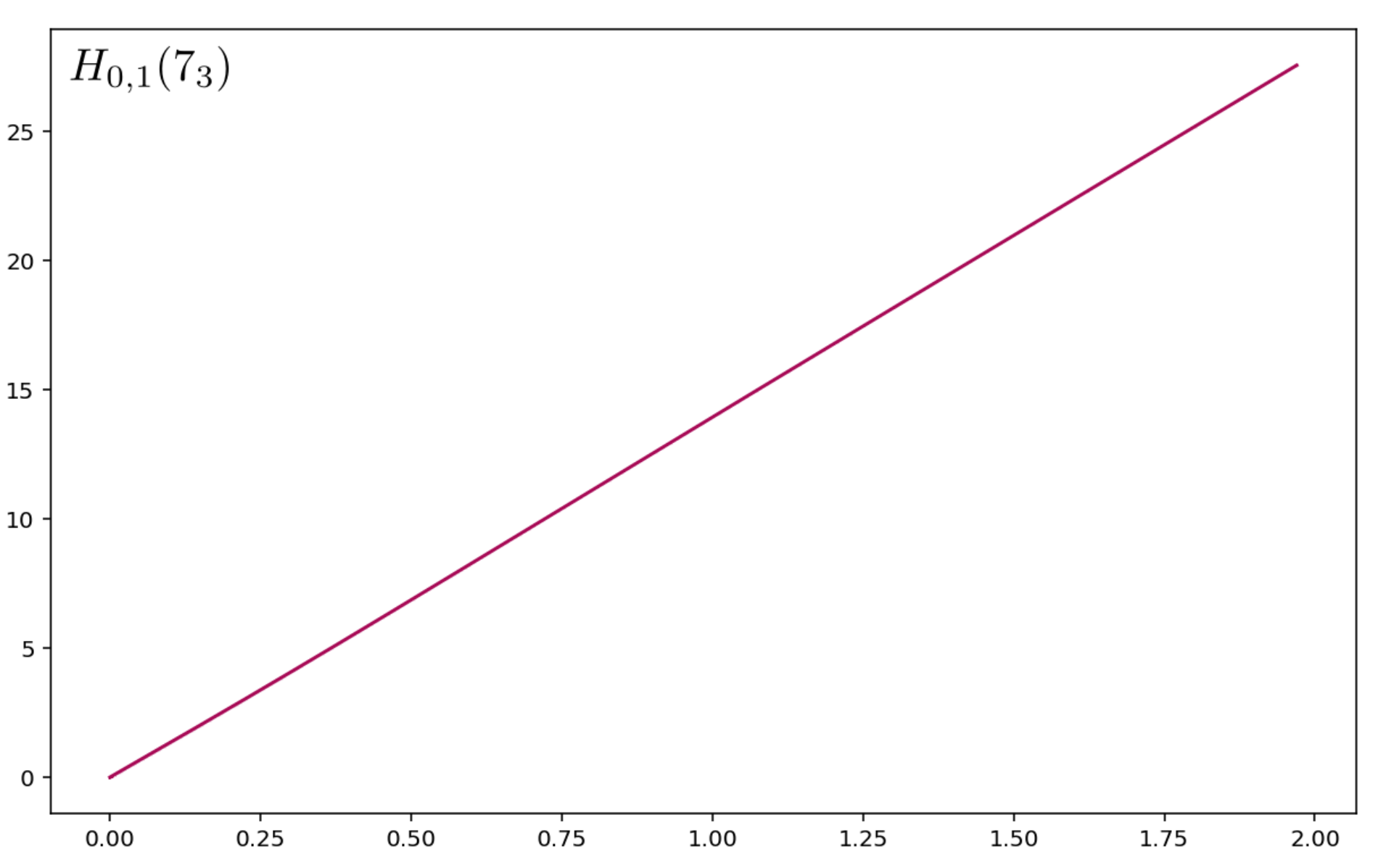}
\includegraphics[width=62mm]{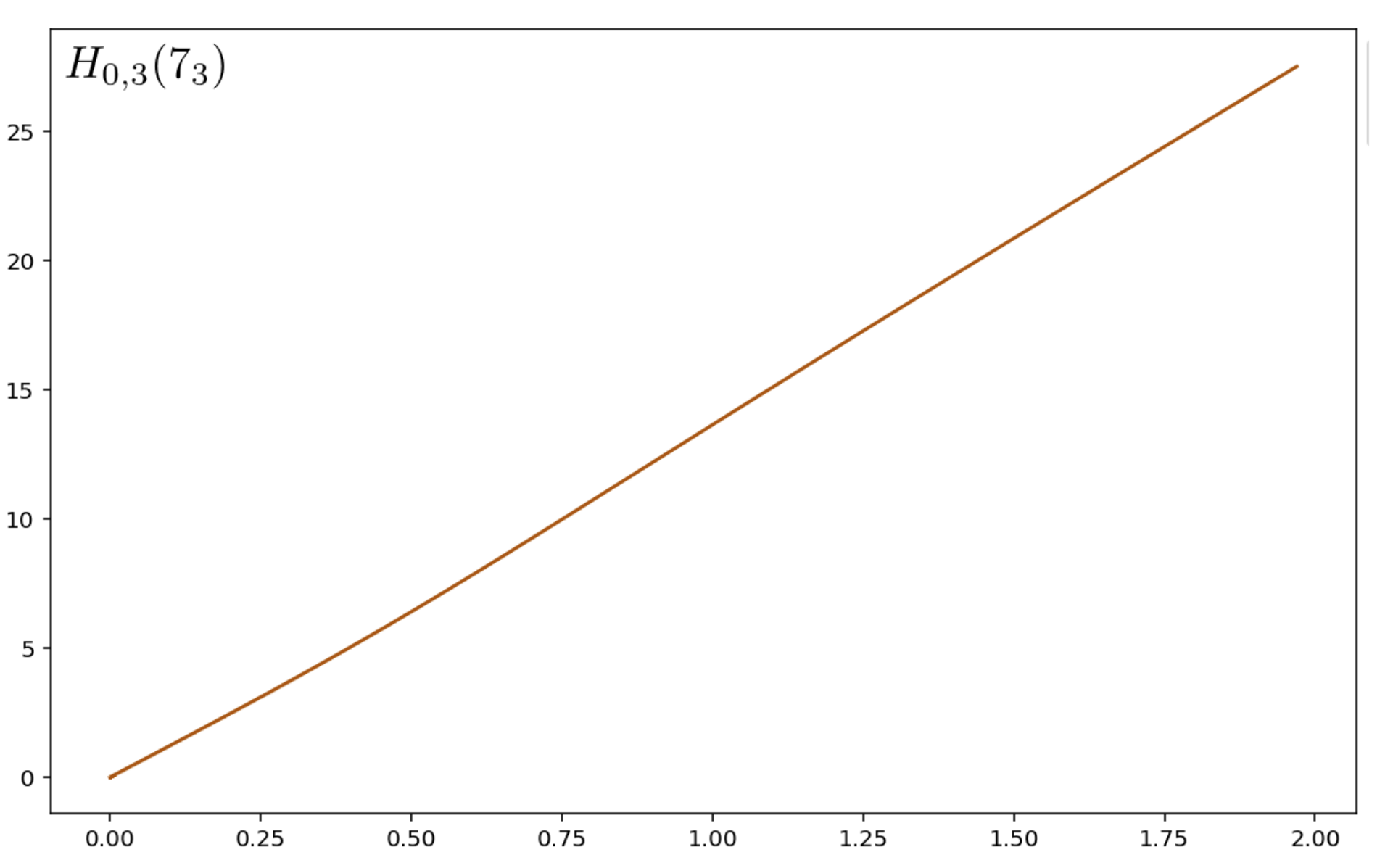}
\caption{$PL_{\widetilde{G}}(7_3)$}
\fnote{The Alexander polynomial of $7_3$ is $2t^4-3t^3+3t^2-3t+2$, which has no real root. But we can see $H_{0,0}(7_3)/(\mathbb{Z}/2\mathbb{Z})$ (figure on top) contains an arc that is different from the $x$-axis, even though this arc does not intersect the $x$-axis. But this arc cannot be predicted by any theorem in this paper. Notice that this arc has two asymptotes of slope $0$ and $6$. So we could predict that Dehn filling of $7_3$ of rational slope in $(-6,0)$ would be orderable.

The snappy command \texttt{normal\_boundary\_slopes()} tells us all the boundary slopes of spun normal surfaces \cite{DG} of $7_3$ are: $0$, $-6$, $-8$, and $-14$. Both arcs in $H_{0,1}(7_3)$ and $H_{0,3}(7_3)$ have asymptote of slope $14$.}
\end{figure}

\subsection{Simple Roots of the Alexander Polynomial}%typical example with simple roots
When the Alexander polynomial $\Delta_M$ of $M$ has a positive root $\xi$, we can draw a point $(\ln(\xi)/2, 0)$ on the $x$-axis and call it an Alexander point. When $\xi$ is a simple root, Lemma \ref{deformation} predicts that there is an arc coming out of the Alexander point $(\ln(\xi)/2, 0)$. Moreover, this Alexander point corresponds to the abelian representation associated to the root $\xi$ of $\Delta_M$, e.g. $\rho_{\alpha}$ as constructed in proof of Lemma \ref{deformation}.
We use large dots to indicate Alexander points in our figures.

In addition to the example of the figure eight knot shown in Figure \ref{figure8}, we will show more holonomy extension loci with Alexander points.
\begin{figure}[H]
\centering
\includegraphics[width=100mm]{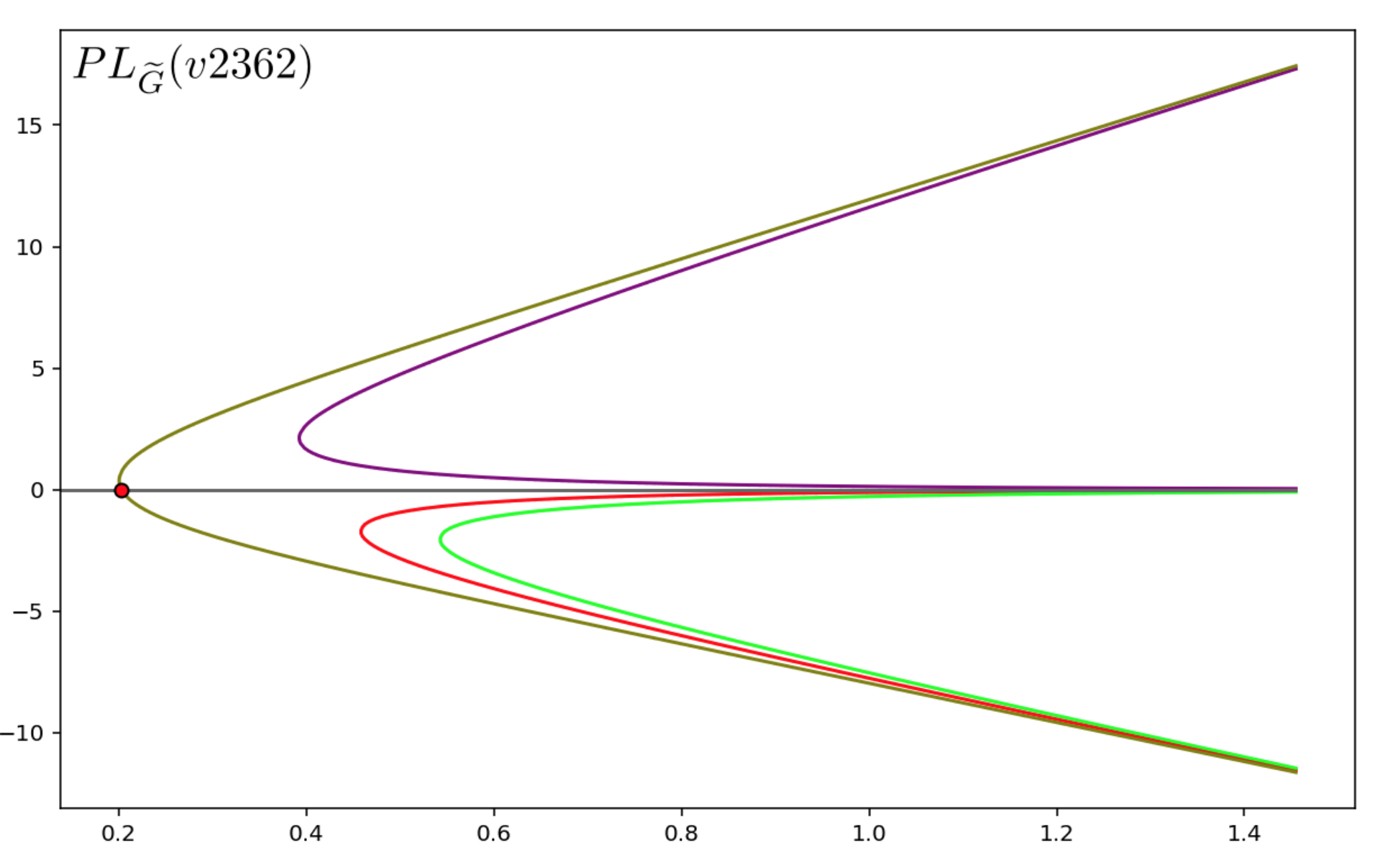}
\caption{$PL_{\widetilde{G}}(v2362)$}
\fnote{This figure is $PL_{\widetilde{G}}(v2362)$, the quotient of the holonomy extension locus of $v2362$, which has only one sheet $H_{0,0}$. The Alexander polynomial of $v2362$ is $6t^2-13t+6$ which has two simple real roots $2/3$ and $3/2$. So we can expect to see the Alexander point $(\frac{1}{2}\ln(\frac{3}{2}),0)$, shown as a red dot in the figure. (The other point $(\frac{1}{2}\ln(\frac{2}{3})=-\frac{1}{2}\ln(\frac{3}{2}),0)$ is mapped to the same point under the quotient action of $\mathbb{Z}/2\mathbb{Z}$.)  We can see in this figure that the arc going through the Alexander point is not tangent to the $x$-axis at the Alexander point. }
\end{figure}

\begin{figure}[H]
\centering
\includegraphics[width=100mm]{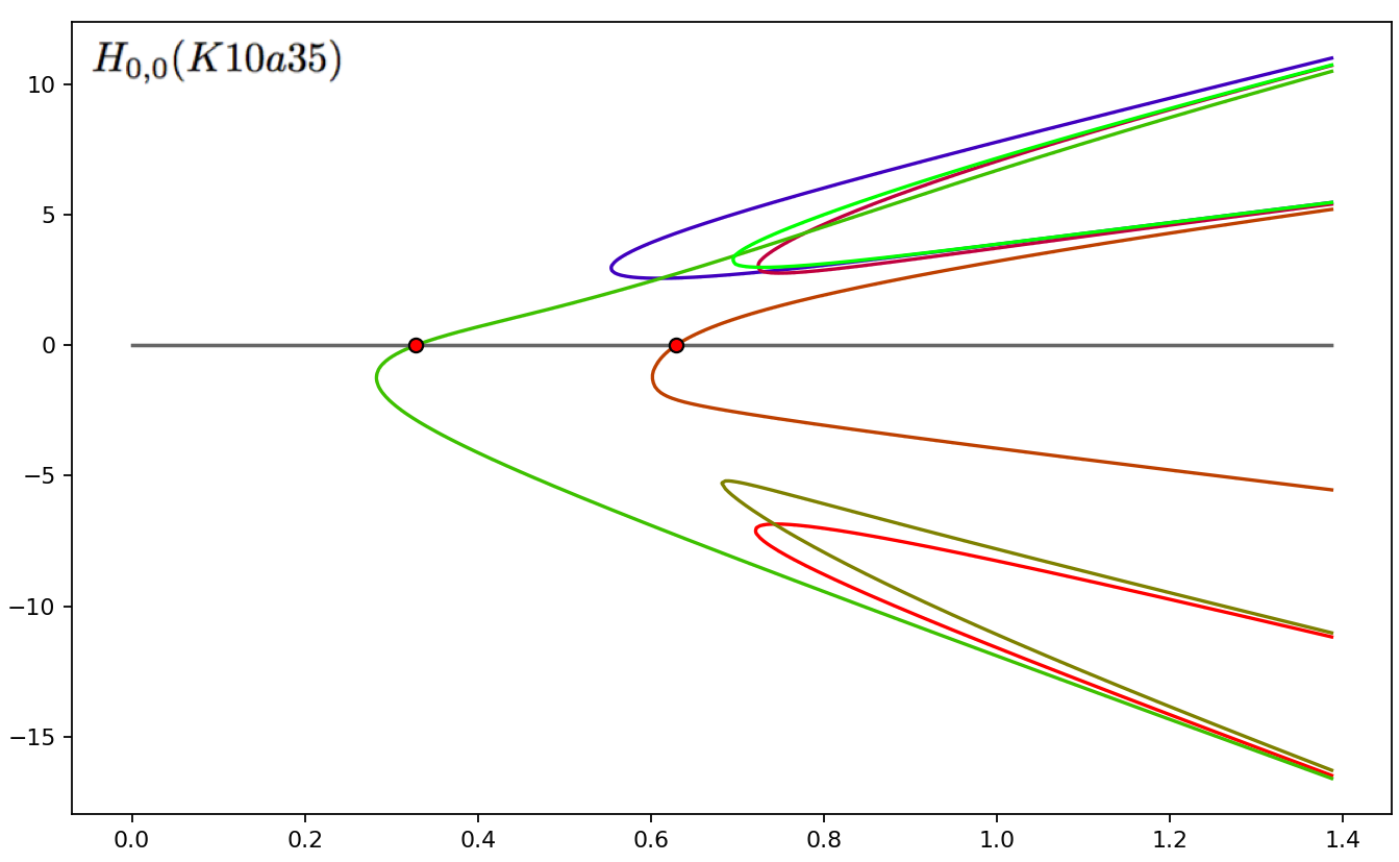}
\caption{$H_{0,0}(K10a35)$}
\fnote{This figure is $H_{0,0}(K10a35)/(\mathbb{Z}/2\mathbb{Z})$. The Alexander polynomial of $K10a35$ has two pairs of simple positive real roots. (The two numbers in each pair are reciprocal so they correspond to the same Alexander point.) So we can expect to see two Alexander points on the x-axis.}
\end{figure}

\begin{figure}[H]
\centering
\includegraphics[width=62mm]{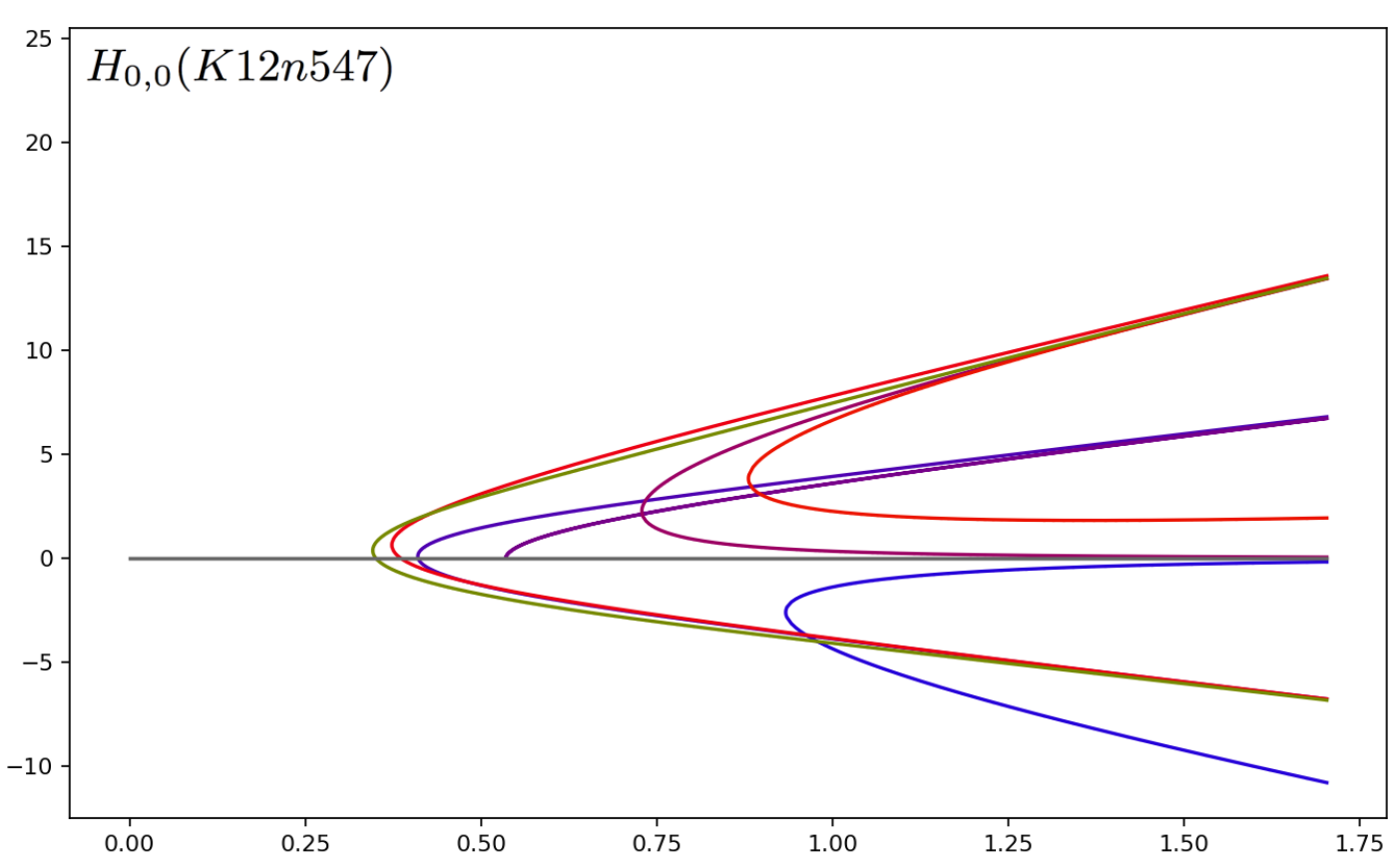}
\includegraphics[width=62mm]{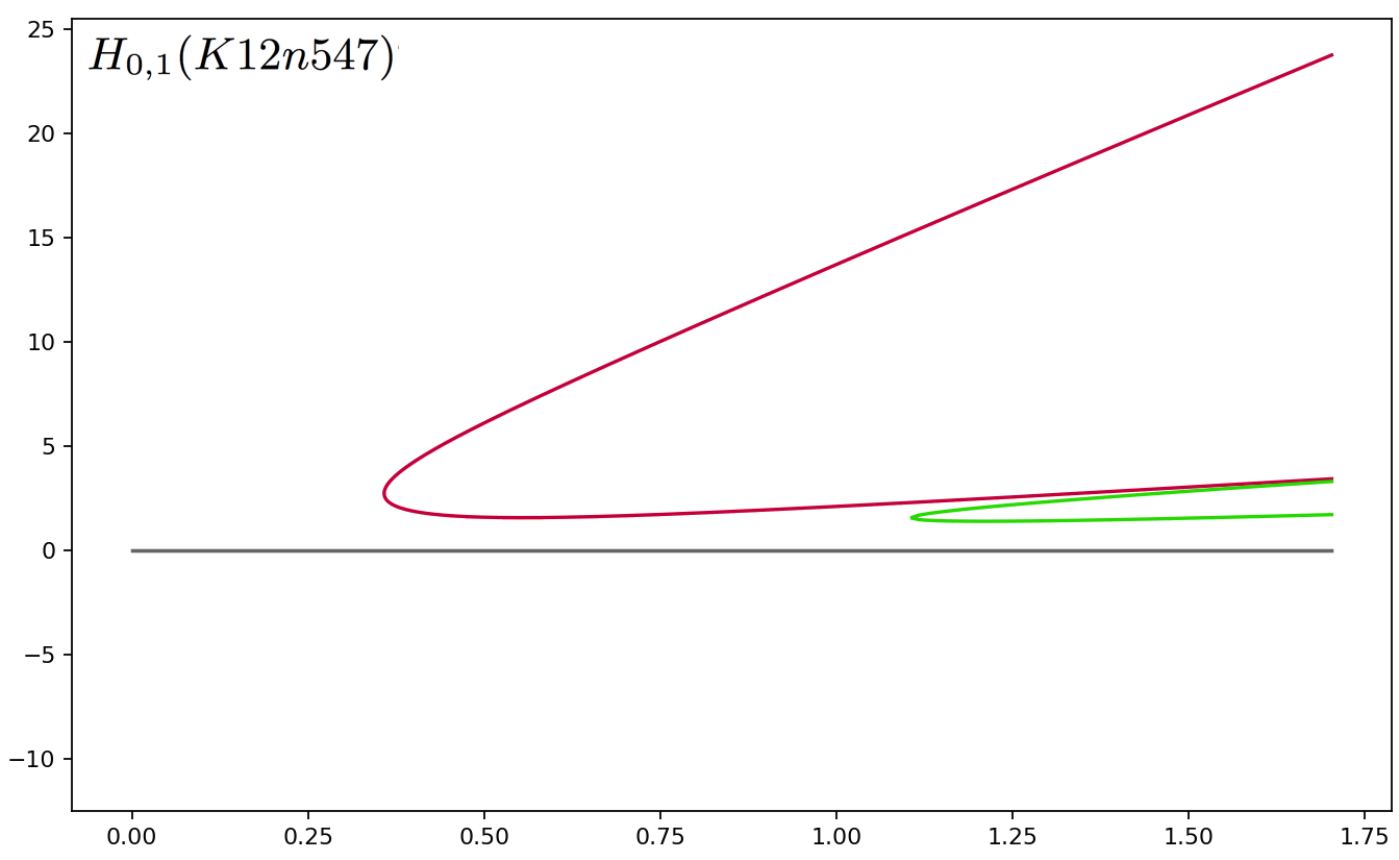}
\caption{$PL_{\widetilde{G}}(K12n547)$}
\fnote{The figure on the left is $H_{0,0}(K12n547)/(\mathbb{Z}/2\mathbb{Z})$. The Alexander polynomial of $K12n547$ has no real root at all. But we can still see quite a few arcs in $H_{0,0}(K12n547)$, many of which even intersect the x-axis, which is a very interesting phenomenon worthy of further exploration. The figure on the right is $H_{0,1}(K12n547)$, where we can see two non intersecting arcs sharing the same asymptote of slope $2$. (The x-axis in the right figure is not contained in $H_{0,1}(K12n547)$. It is included only to show no arc in $H_{0,1}(K12n547)$ intersects the x-axis.)}
\end{figure}

\subsection{Multiple Roots of the Alexander Polynomial}

\begin{figure}[H]
\centering
\includegraphics[width=100mm]{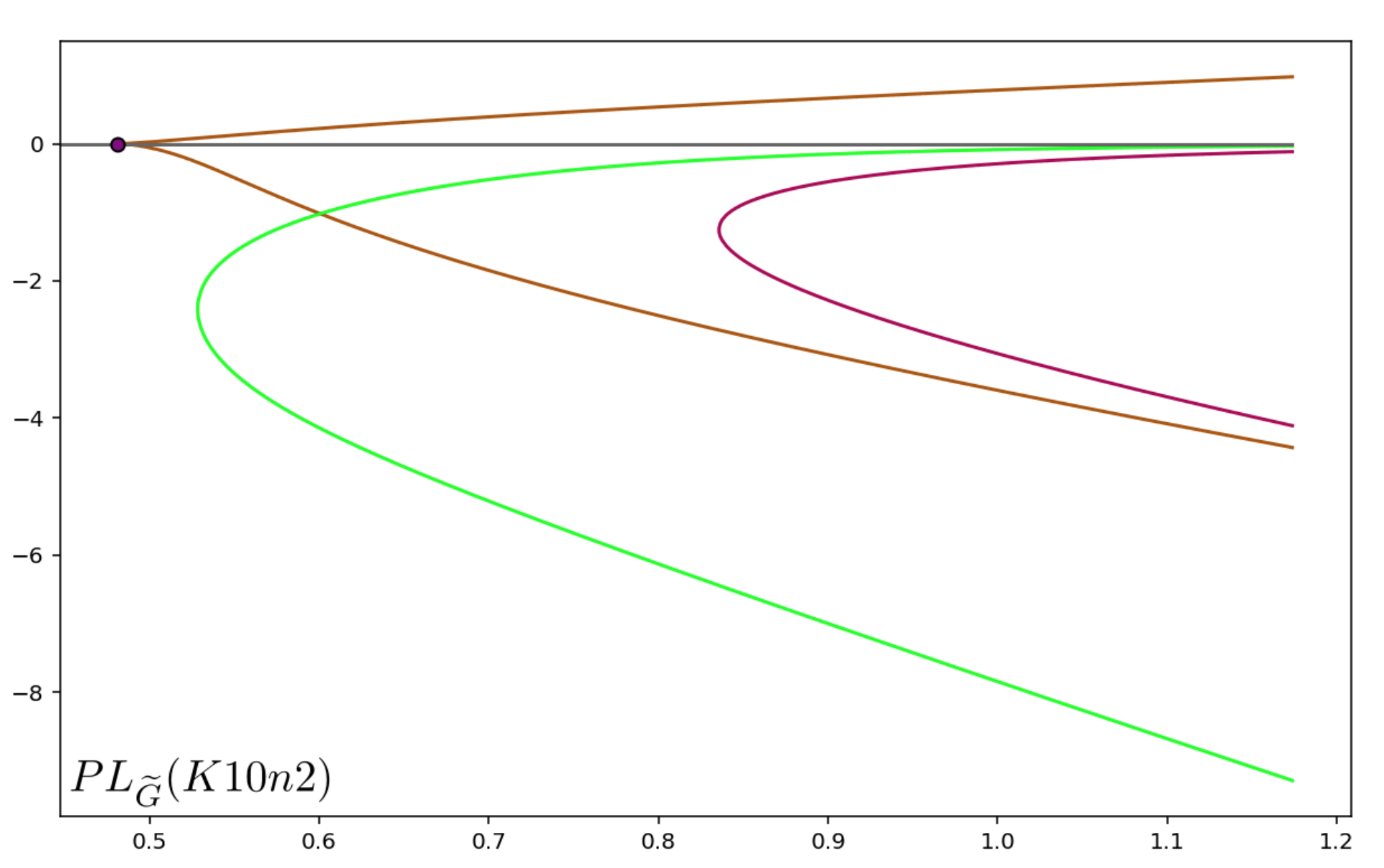}
\caption{$PL_{\widetilde{G}}(K10n2)$}
\fnote{This figure is $PL_{\widetilde{G}}(K10n2)$, the quotient of the holonomy extension locus of $K10n2$. It contains only one sheet, the quotient locus $H_{0,0}/(\mathbb{Z}/2\mathbb{Z})$. The Alexander polynomial of $K10n2$ has a pair of positive real double roots so there is an Alexander point. We can see that the two arcs are tangent to the $x$-axis at the Alexander point. }
\label{d1}
\end{figure}

\begin{figure}[H]
\centering
\includegraphics[width=130mm]{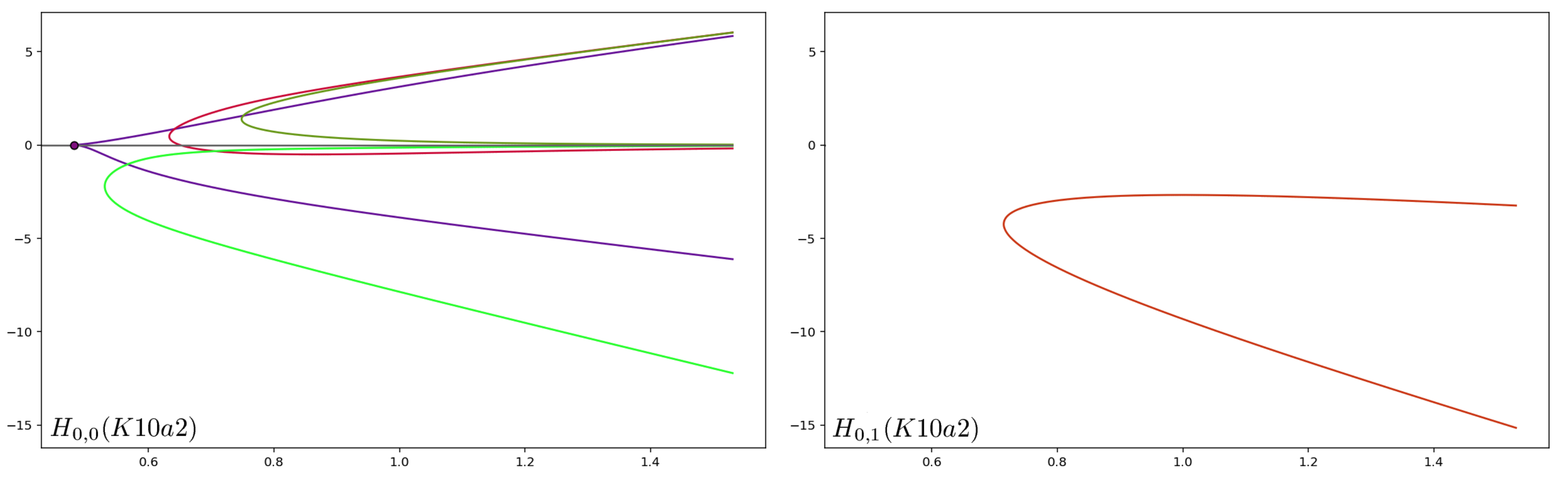}
\caption{$PL_{\widetilde{G}}(K10a2)$}
\fnote{The Alexander polynomial of $K10a2$ has a pair of positive real double roots. We can see that two arcs in $H_{0,0}(K10a2)$ in the left figure are tangent to the $x$-axis at the Alexander point.}
\label{d2}
\end{figure}

The holonomy extension loci of $K10n2$ (Figure \ref{d1}) and $K10a2$ (Figure \ref{d2}) show typical patterns of $\mathbb{Q}$-homology solid tori whose Alexander polynomials have double real roots: they all have arcs tangent to the $x$-axis at the corresponding Alexander point.

The manifold $K9a37$ in our next example also has Alexander polynomial with double roots. However the local picture of its holonomy extension locus at the Alexander point is quite different from what we see in Figure \ref{d1} and \ref{d2}.  %and \ref{d3}.
See Figure \ref{K9a37} for explanations. 

The holonomy extension locus of $K9a37$  has some interesting phenomena, which are shown in Figure \ref{K9a37}. Part of the red curve (second curve from the bottom) marked with `x's  means that the point on the marked part of the curve comes from a PSL$_2\mathbb{C}$ representation $\rho$ that is not PSL$_2\mathbb{R}$ even though $\rho|_{\partial M}$ is a PSL$_2\mathbb{R}$ representation. So these points do not belong to the holonomy extension locus. (The small dots on the curves simply means this point comes from a PSL$_2\mathbb{R}$ representation.) From this example, we can see that an arc in a holonomy extension locus can end at a point that is not the infinity, Alexander point or parabolic point. We guess such a point could be a Tillmann point (see \cite{CD16} end of Section 5 for definition).

\begin{figure}[H]
\centering
\includegraphics[width=100mm]{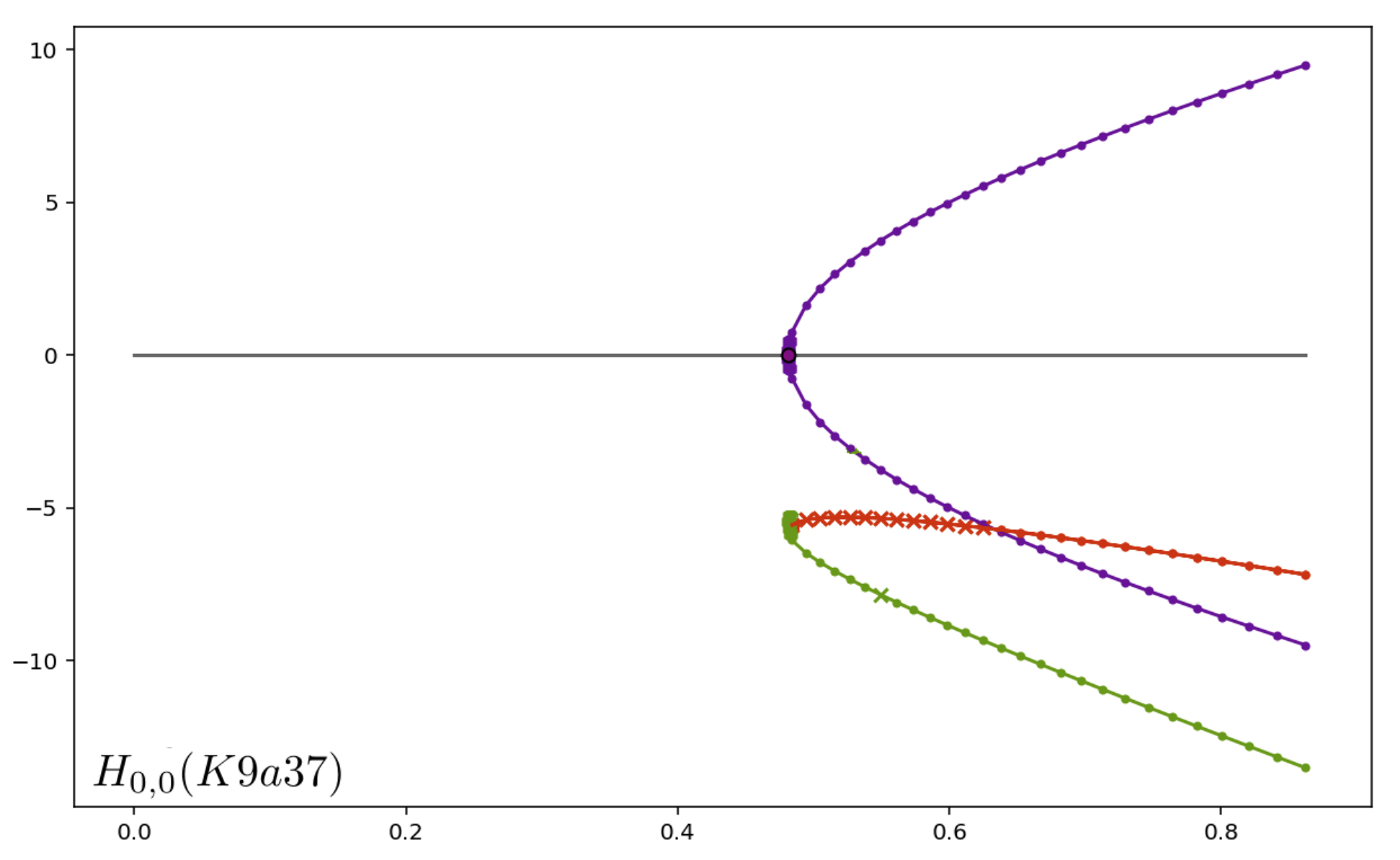}
\caption{$H_{0,0}(K9a37)$}
\label{K9a37}
\fnote{The Alexander polynomial of $K9a37$ has two positive real double roots. Figure \ref{K9a37} is $H_{0,0}$ of the holonomy extension locus of $K9a37$. (To be precise, we still need to remove a small segment of arc on the red curve (second curve from the bottom) to get the actual $H_{0,0}(K9a37)$. ) We can see that there is an arc coming out of the Alexander point in both directions but not tangent to the $x$-axis.\\
\textbf{Remark:} If you run the graphing program directly, the graph you see is slightly different from this figure. In fact, there is an extra arc which does not belong to $H_{0,0}(K9a37)$ but appears tangent to the bottom arc (shown in green) in $H_{0,0}(K9a37)$ and our current graphing program is unable to separate it out automatically. So we have to remove the extra arc by hand.\\
\textbf{Remark 2:} In addition to issues with graphing like unseparated curves and Tillmann points as mentioned above, we also spotted missing components. In the above example $K9a37$, we know an arc in $H_{0,2}(K9a37)$ is missing from our figure. In their graphing program, Culler and Dunfield use gluing varieties rather than character varieties to simplify computation. Some of the graphing issues might be caused by this. Check the end of Section 5 of \cite{CD16} for more details about computation and graphing issues.
}

\end{figure}

The statement of Lemma \ref{deformation} requires the root of the Alexander polynomial to be simple.
When we have a root that is not simple, we expect to see an example where there is no arc coming out of the corresponding Alexander point at all, as this is what happened in the translation extension locus in Figure 10 of Section 5 of \cite{CD16}.
However, we were not able to find such an example at this moment as the graphing program is still unfinished and we only have very limited number of samples.

Even though the graph of the holonomy extension locus cannot function as a precise proof, as it comes from numerical computation. It is still very enlightening in the sense that, if $L_r$ does not intersect the graph of $H_{0,0}(M)$ for any rational $r$ inside some interval, then either the Dehn filling $M(r)$ is not orderable or its orderability could not be proven using the method of representation into $\widetilde{G}$ and a larger subgroup of Homeo$^+(\mathbb{R})$ needs to be taken into consideration.

\section{Alexander polynomials and orderability}\label{5}
In this section, we prove Theorem \ref{Thm_Alex}. To state the theorem, we will need some definitions from \cite{CD16}.
We say a compact 3-manifold $Y$ has few characters if each positive dimensional component of the PSL$_2\mathbb{C}$ character variety of $Y$ consists entirely of characters of reducible representations. An irreducible $\mathbb{Q}$-homology solid torus $M$ is called longitudinally rigid when its Dehn filling $M(0)$ along the homological longitude has few characters.

The author learnt from a private conversation that this theorem was also proved independently by Steven Boyer.

\begin{theorem}\label{Thm_Alex}%[7.1]
Suppose $M$ is the exterior of a knot in a $\mathbb{Q}$-homology $3$-sphere that is longitudinal rigid. If the Alexander polynomial $\Delta_M$ of $M$ has a simple positive real root $\xi\neq 1$, then there exists a nonempty interval $(-a, 0]$ or $[0, a)$ such that for every rational $r$ in the interval, the Dehn filling $M(r)$ is orderable.
\end{theorem}

The following lemma is the key to the proof of Theorem \ref{Thm_Alex}. In Theorem \ref{Thm_Alex}, we need $M$ to be longitudinal rigidity to ensure that the path constructed in Lemma \ref{deformation} is not contained in the x-axis when mapped to the holonomy extension locus.

\begin{lemma}\label{deformation}%[7.3]
Suppose $M$ is an irreducible $\mathbb{Q}$-homology solid torus. If $\xi\neq 1$ is a simple positive real root of the Alexander polynomial, then there exists an analytic path $\rho_t: [-1,1]\rightarrow R_G(M)$ where:
\begin{itemize}
\item[(a)] The representations $\rho_t$ are irreducible over PSL$_2\mathbb{C}$ for $t\neq 0$.
\item[(b)] The corresponding path $[\rho_t]$ of characters in $X_G(M)$ is also a nonconstant analytic path.
\item[(c)] $\text{tr}^2_{\gamma}(\rho_t)$ is nonconstant in $t$ for some $\gamma\in \pi_1(\partial M)$.
\end{itemize}
\end{lemma}

To study the smoothness of a point on the character variety, we need to study the Zariski tangent space at that point.

\begin{definition}\cite[3.1.3]{porti_french}
Suppose $V$ is an affine algebraic variety in $\mathbb{C}^n$. Let $I(V)=\{ f\in \mathbb{C}[x_1,\ldots, x_n]|f(x)=0, \forall x\in V\}$ be the ideal of $V$. Define the Zariski tangent space to $V$ at $p$ to be the vector space of derivatives of polynomials:
\begin{displaymath}
T_p^{\text{Zar}}(V)=\{\frac{d\gamma}{dt}|_{t=0}\in \mathbb{C}^n|\gamma\in(\mathbb{C}[t])^n, \gamma(0)=p\ s.t.\  f\circ\gamma\in t^2\mathbb{C}[t]\  \forall f\in I(V) \}.
\end{displaymath}
\end{definition}
A point $p$ on $V$ is called smooth if the dimension of $T_p^{\text{Zar}}(V)$ is equal to the dimension of the component of $V$ which $p$ lies on.

Let $\Gamma$ be a group and let $\rho: \Gamma\rightarrow \text{PSL}_2\mathbb{C}$ be a representation. Then we can turn the Lie algebra $\mathfrak{sl}_2(\mathbb{C})$ into a $\Gamma$ module via the adjoint representation, which means taking conjugation $g\cdot a:=\rho(g)a\rho(g)^{-1}$. Denote this $\Gamma$ module by $\mathfrak{sl}_2(\mathbb{C})_{\rho}$. Then the Zariski tangent space of the character variety $X_{\text{PSL}_2\mathbb{C}}(\Gamma)$ at $[\rho]$ is a subspace of the cohomology $H^1(\Gamma; \mathfrak{sl}_2(\mathbb{C})_{\rho})$. \cite[Proposition 3.5]{porti_french}

\begin{proof}[Proof of Lemma \ref{Thm_Alex}]% of lemma \ref{deformation}
First I prove (a) and (b).

As in Proposition 10.2 of \cite{HP2}, let $\alpha: \pi_1(M)\rightarrow \mathbb{R}_+=(\mathbb{R}>0)$ be a representation such that $\alpha$ factors through $H_1(M; \mathbb{Z})_{\text{free}}\cong \mathbb{Z}$ and takes a generator of $H_1(M; \mathbb{Z})_{\text{free}}$ to $\xi$.
Let $\rho_\alpha: \pi_1(M)\rightarrow \text{PSL}_2\mathbb{R}$ be the associated diagonal representation given by
\begin{displaymath}
\rho_\alpha=\pm \left[\begin{matrix}\alpha^{1/2}(\gamma) & 0\\0 &\alpha^{-1/2}(\gamma)\end{matrix}\right], \text{ where $\alpha^{1/2}(\gamma)$ is the positive square root of } \alpha(\gamma).
\end{displaymath}
Then $\chi_{\alpha}=\text{tr}^2(\rho_\alpha)$ is real valued, as $\alpha(\gamma)+1/\alpha(\gamma)+2\in \mathbb{R}$, $\forall \gamma\in\pi_1(M)$. Since Im$(\alpha)$ is contained in $\mathbb{R}_+$ but not in $\{\pm 1\}$, Im$(\rho_{\alpha})$ is contained in PGL$_2(\mathbb{R})$ and in fact in PSL$_2\mathbb{R}$.
Next, we carry out the computation of obstruction in the real setting. As $\mathfrak{sl}_2(\mathbb{C})$ is the complexification of $\mathfrak{sl}_2(\mathbb{R})$, the Lie algebra of SL$_2\mathbb{R}$, we have the corresponding isomorphism of cohomology groups.
\begin{displaymath}
H^*(\pi_1(M); \mathfrak{sl}_2(\mathbb{C})_{\rho_{\alpha}})=H^*(\pi_1(M); \mathfrak{sl}_2(\mathbb{R})_{\rho_{\alpha}})\otimes_{\mathbb{R}}\mathbb{C}.
\end{displaymath}
So computations with complex variety $X(M)$ in the proof of \cite[Theorem 1.3]{HP2} can be carried out in the real case.
It follows that the tangent space to $X_G(M)$ at $\chi_{\alpha}$ is $H^*(\pi_1(M); \mathbb{R}_+\oplus\mathbb{R}_-)//\mathbb{R}^*\cong \mathbb{R}$ and thus $\chi_{\alpha}$ is a smooth point. Carrying out the computation of obstructions in the real setting, we are able to show that $d_++d_-\in H^1(\pi_1(M); \mathfrak{sl}_2(\mathbb{R})_{\rho_{\alpha}})$ can be integrated to an analytic path $\rho_t: [-1,1]
\rightarrow R_{G}(M)$ with $\rho_0=\rho_{\alpha}$ and $\rho_t$ irreducible over PSL$_2\mathbb{C}$ for $t\neq 0$.
So $\chi_{\alpha}$ is contained in a curve containing characters of irreducible PSL$_2\mathbb{R}$ representations, which gives (a).

The path $[\rho_t]\subset X_G(M)$ is nonconstant because $\rho_t$ is irreducible whenever $t\neq 0$ and thus cannot have same character as the reducible representation $\rho_0$, proving part (b).

Next, we will prove part (c).
In fact the existence of $\gamma\in \pi_1(\partial M)$ such that $\text{tr}^2_{\gamma}(\rho_t)$ is nonconstant in $t$ is proved similarly as in \cite[Lemma 7.3 (4)]{CD16}.
We first construct a nonabelian representation $\rho^+\in R_G(M)$ which corresponds to $[\rho_{\alpha}]$ in $X_G(M)$. Then the Zariski tangent space of $X_G(M)$ at $[\rho_{\alpha}]$ can be identified with $H^1(M; \mathfrak{sl}_2(\mathbb{R})_{\rho^+})$, while the Zariski tangent space of $X_G(\partial M)$ at $[\rho^+\circ \iota]$ can be identified with $H^1(\partial M; \mathfrak{sl}_2(\mathbb{R})_{\rho^+})$. So the proof of (c) boils down to showing the injectivity of $\iota^*: H^1(M; \mathfrak{sl}_2(\mathbb{R})_{\rho^+})\rightarrow H^1(\partial M; \mathfrak{sl}_2(\mathbb{R})_{\rho^+})$. See \cite[Lemma 7.3 (4)]{CD16} for more details.

\end{proof}

We will also need the following property of closed $3$ manifolds with few characters.

\begin{lemma}\label{irred}
Suppose $Y$ is a closed $3$ manifold with $H_1(Y, \mathbb{Q})=\mathbb{Q}$. If $Y$ has few characters, then $Y$ is irreducible.
\end{lemma}

\begin{proof}
Prove by contradiction. If $Y$ is reducible, then we can decompose it as a connected sum $Y_1\sharp Y_2$, where $H_1(Y_1, \mathbb{Q})=\mathbb{Q}$ and $Y_2$ is a $\mathbb{Q}$HS. So $\pi_1(Y)=\pi_1(Y_1)\ast \pi_1(Y_2)$. We want to use PSL$_2\mathbb{C}$ representations of $Y_1$ and $Y_2$ to construct a dimension one component of PSL$_2\mathbb{C}$ character variety of $Y$ containing an irreducible representation so that it contradicts the assumption that $Y$ has few characters.
As $H_1(Y_1, \mathbb{Z})=\mathbb{Z}\oplus (\text{possible torsion})$, we can construct a nontrivial abelian PSL$_2\mathbb{C}$ representation $\rho_1$ of $Y_1$ by composing $\pi_1(Y_1)\twoheadrightarrow \mathbb{Z}$ and $\mathbb{Z}\hookrightarrow \text{PSL}_2\mathbb{C}$.
For $Y_2$, there are two cases. If $H_1(Y_2, \mathbb{Z})$ contains a cyclic subgroup $H$, then similarly we can construct a nontrivial abelian PSL$_2\mathbb{C}$ representation $\rho_2$ of $Y_2$ by composing $\pi_1(Y_2)\twoheadrightarrow H$ and $H\hookrightarrow \text{PSL}_2\mathbb{C}$. If $Y_2$ is actually a $\mathbb{Z}$HS, then by Theorem 9.4 of \cite{Zentner}, there is an irreducible SL$_2\mathbb{C}$ representation $\rho_2$ of $\pi_1(Y_2)$. Moreover we can make $\rho_2$ an irreducible PSL$_2\mathbb{C}$ representation by simply projecting to PSL$_2\mathbb{C}$.
So we can construct a set of PSL$_2\mathbb{C}$ representations $\rho_P=\rho_1 \ast P\rho_2P^{-1}$ of $Y$, where $P$ is any matrix in PSL$_2\mathbb{C}$. These representations are not conjugate to each other as long as they have different $P$ and at least one of them is irreducible as we can vary $P$ so that $\rho_1$ and $P\rho_2P^{-1}$ are not upper triangular at the same time.
\end{proof}

Now we can prove Theorem \ref{Thm_Alex}.
\begin{proof}[Proof of Theorem \ref{Thm_Alex}]
The key idea of proof is to show that a nontrivial $\widetilde{G}$ representation of $\pi_1(M(r))$ exists for rational $r$ in some small interval containing $0$, which mainly uses Lemma \ref{main lemma} and Lemma \ref{Thm_Alex}.

Let $\rho_t$ be the associated path in $R_G(M)$ given by Lemma \ref{deformation}.
As $\rho_0$ factors through $H_1(M; \mathbb{Z})_{\text{free}}\cong \mathbb{Z}$, we can lift it to $\widetilde{G}$ and its lift $\widetilde{\rho_0}$ also factors through $H_1(M; \mathbb{Z})_{\text{free}}$. Hence trans$(\widetilde{\rho_0}(\lambda))=0$ and we can modify $\widetilde{\rho_0}$ so that trans$(\widetilde{\rho_0}(\mu))=0$. Then $\widetilde{\rho_0}$ is mapped to a point on the horizontal axis of $H_{0,0}(M)$ as $\rho_0(\lambda)=I$. The $x$ coordinate of $\widetilde{\rho_0}$, $\ln(|\xi|)$ is nonzero as $\xi\neq \pm 1$. 

As $\rho_0$ lifts, we can extend this lift to a continuous path $\widetilde{\rho}_t$ in $R_{\widetilde{G}}(M)$. Moreover, we can assume $\widetilde{\rho}_t$ is actually in $R^{\text{aug}}_{\widetilde{G}}(M)$, as fixed points of $\widetilde{\rho}_t(\pi_1(\partial M))$ also vary continuously with $t$.

Let $k$ be the index of $\left<\iota_*(\mu)\right>$ in $H_1(M, \mathbb{Z})_{\text{free}}$, where $\iota: \partial M \rightarrow M$ is the inclusion. By construction tr$^2_{\mu}(\widetilde{\rho_0})=\xi^k+2+\xi^{-k}> 4$, so there exists $\varepsilon>0$ such that tr$^2_{\mu}(\widetilde{\rho_t})\geq 4$ for $t\in [-\varepsilon, \varepsilon]$.
As $\rho_t(\mu)$ is hyperbolic, $\rho_t(\lambda)$ is also hyperbolic. Therefore $\rho_t$ is a path in $PH_{G}(M)$ and $\widetilde{\rho_t}$ is a path in $PH_{\widetilde{G}}(M)$.

Then we can build a path $A$ by composing $\widetilde{\rho}_t$ with EV$\circ \iota^*: PH_{\widetilde{G}}(M)\rightarrow HL_{\widetilde{G}}(M)$. That the path $A$ is nonconstant follows from Lemma \ref{Thm_Alex}. Moreover, it is not contained in $x$-axis $L_0$. If it were contained in the $x$-axis, then $\rho_t(\lambda)=I$ as $\rho_t(\lambda)$ is always hyperbolic or trivial. So each $\rho_t$ factors through a representation of the $0$ filling $M(0)$. Therefore $[\rho_t]$ must lie in a component of $X(M(0))$ of dimension at least $1$, contradicting the assumption that $M$ is longitudinally rigid.

Since all points in $A$ come from actual $\widetilde{G}$ representations, there is no ideal point in $A$. As all but at most three Dehn fillings of a knot complement are reducible \cite[Theorem 1.2]{GL1}, we can shrink $A$ if necessary so that none of the Dehn fillings involved is reducible.
The only parabolic point in $H_{0,0}(M)$ is the origin so $A$ contains no parabolic point.
Applying Lemma $\ref{main lemma}$, we get an interval $(0, a)$ or $(-a, 0)$ of orderable Dehn fillings.

Finally, we show $M(0)$ is orderable. The first Betti number of $M(0)$ is $1$ as rational homology groups of $M(0)$ are the same as $S^2\times S^1$. The irreducibility of $M(0)$ follows from Lemma \ref{irred}. So we can apply Theorem 1.1 of \cite{BRW} and show that $\pi_1(M(0))$ is left-orderable, completing the proof of the theorem.

\end{proof}

\begin{remark}
In our numerical computation, an interval of the form $(-a, b)$ with $a,b>0$ is expected to hold in Theorem \ref{Thm_Alex}. This could be observed from the graphs of holonomy extension loci in the previous section. However the author was not able to prove it. See Section \ref{problems} for details.
\end{remark}

\section{Real embeddings of trace fields and orderability}
In this section, we use a different assumption for the manifolds we study, and prove Theorem \ref{Galois}.

Let $Y$ be a closed hyperbolic 3-manifold with fundamental group $\Gamma$. Let $\rho_{hyp}: \Gamma\rightarrow \text{PSL}_2\mathbb{C}$ be the holonomy representation of $Y$. The trace field $K=\mathbb{Q}(\text{tr} \Gamma)$ of $\rho_{hyp}$ is the subfield of $\mathbb{C}$ generated over $\mathbb{Q}$ by the traces of lifts to SL$_2\mathbb{C}$ of all elements in $\rho_{hyp}(\Gamma)$. It is a number field by \cite[Theorem 3.1.2]{Arithmetic}. Assume we have a real embedding $\sigma$ of the trace field $K$ into $\mathbb{R}$.

Define the associated quaternion algebra to be $D=\{\Sigma a_i\gamma_i|a_i\in K, \gamma_i\in \rho_{hyp}(\Gamma)\}$.
To say $D$ splits at the real embedding $\sigma$ means $D\otimes_{\sigma}\mathbb{R}\cong M_2(\mathbb{R})$, which implies that we can conjugate $\Gamma$ into PSL$_2\mathbb{R}$. So we get a Galois conjugate representation $\overline{\rho}: \Gamma\rightarrow \text{PSL}_2\mathbb{R}$. See Section 2.1 and 2.7 of \cite{Arithmetic} for more details.

The following conjecture is due to Dunfield.

\begin{conjecture}
Suppose $M$ is a hyperbolic $\mathbb{Z}$ homology solid torus. Assume the longitudinal filling $M(0)$ is hyperbolic and its holonomy representation has a trace field with a real embedding at which the associated quaternion algebra splits. Then every Dehn filling $M(r)$ with rational $r$ in an interval $(-a,a)$ is orderable.
\end{conjecture}

By adding some extra conditions, I am able to prove the following result.

\begin{theorem}\label{Galois}
Suppose $M$ is a hyperbolic $\mathbb{Z}$-homology solid torus. Assume the longitudinal filling $M(0)$ is a hyperbolic mapping torus of a homeomorphism of a genus $2$ orientable surface and its holonomy representation has a trace field with a real embedding at which the associated quaternion algebra splits. Then every Dehn filling $M(r)$ with rational $r$ in an interval $(-a, 0]$ or $[0, a)$ is orderable.
\end{theorem}

First let us fix some notations. Denote the holonomy representation of the hyperbolic manifold $M(0)$ by $\rho_{hyp}: \pi_1(M(0))\longrightarrow \text{PSL}_2\mathbb{C}$ and the projection map $p: \pi_1(M)\rightarrow \pi_1(M(0))$.
The composition $\rho_M=p\circ \rho_{hyp}$ has kernel normally generated by the longitude $\lambda$.
The Galois conjugate of $\rho_{M}$ is denoted by $\rho_0$. It is also the Galois conjugate of $\rho_{hyp}$ composed with $p$.
Denote $\rho_V: \pi_1(V)\longrightarrow \text{PSL}_2\mathbb{C}$ the induced representation of $\rho_{hyp}$ on $V=S^1\times D^2 \subset M(0)$, 
and $\rho_{T^2}: \pi_1(T^2)\longrightarrow \text{PSL}_2\mathbb{C}$ the induced representation of $\rho_{hyp}$ on $\partial M=T^2$.

Weil's infinitesimal rigidity in the compact case \cite{Weil}, which is stated as follows, is the key to the proof of Theorem \ref{Galois}.
\begin{theorem}\label{Weil}
Let $M$ be a compact 3-manifold with torus boundary whose interior admits a hyperbolic structure with finite volume, then
$H^1(M(0),\mathfrak{sl}_2(\mathbb{C})_{\rho_{hyp}})=0$. (See also \cite[Section 3.3.3]{porti_french}\cite{FPorti})
\end{theorem}
The reference \cite{porti_french} works with SL$_2\mathbb{C}$ rather than $\text{PSL}_2\mathbb{C}$ character varieties.
So to apply the argument in \cite{porti_french}, we will lift $\text{PSL}_2\mathbb{R}$ representations to SL$_2\mathbb{R}$ when necessary. That they always lift is guaranteed by \cite[Proposition 3.1.1]{CS}.

The proof of Theorem \ref{Galois} relies on the following lemma whose proof is based on Weil's theorem.

\begin{lemma}\label{galois_smooth}%[8.3]
Suppose $\rho_0$ is defined as above. Then there exists an arc $c$ in $R_G(M)$ such that
\begin{itemize}
\item[(a)]  $c \ni \rho_0$ is a smooth point of $R_G(M)$.
\item[(b)]  $\text{tr}^2_{\gamma}$ is a local parameter of the arc $c$ near $\rho_0$ (see e.g. \cite[2.1]{shaf1} for def.), where $\gamma\in \pi_1(\partial M)$ is some primitive element different from the longitude $\lambda$.
\end{itemize}
\end{lemma}

\begin{proof}

(a) First, let us prove that $\rho_0$ is a smooth point of $R_G(M)$.
We compute the Mayer-Vietoris sequence for cohomology with local coefficients, associated to the decomposition $M(0)=M\cup_{\partial M}V$.

\begin{align*}
        \cdots &\rightarrow H^1(M(0); \mathfrak{sl}_2(\mathbb{C})_{\rho_{hyp}})\\
        &\rightarrow H^1(V; \mathfrak{sl}_2(\mathbb{C})_{\rho_{V}})\oplus H^1(M; \mathfrak{sl}_2(\mathbb{C})_{\rho_M}) \rightarrow H^1(T^2; \mathfrak{sl}_2(\mathbb{C})_{\rho_{T^2}})\rightarrow \\
        &\rightarrow H^2(M(0); \mathfrak{sl}_2(\mathbb{C})_{\rho_{hyp}}) \rightarrow\cdots
    \end{align*}

The first term $H^1(M(0); \mathfrak{sl}_2(\mathbb{C})_{\rho_{hyp}})=0$ follows from Weil's infinitesimal rigidity Theorem \ref{Weil}.
So $H^1(V; \mathfrak{sl}_2(\mathbb{C})_{\rho_{V}})\oplus H^1(M; \mathfrak{sl}_2(\mathbb{C})_{\rho_M}) \rightarrow H^1(T^2; \mathfrak{sl}_2(\mathbb{C})_{\rho_{T^2}})$ is an injection.
To see that it is actually an isomorphism, note that by Poincar\'{e} duality $H^2(M(0); \mathfrak{sl}_2(\mathbb{C})_{\rho_{hyp}})\cong H^1(M(0); \mathfrak{sl}_2(\mathbb{C})_{\rho_{hyp}})=0$.

Let $X_c(M)$ be the component of $X(M)$ containing $[\rho_M]$.
As $\rho_V$ and $\rho_{T^2}$ are nontrivial, then by \cite[Theorem 1.1 (i)]{BN}, we have $\dim_{\mathbb{C}} H^1(V;  \mathfrak{sl}_2(\mathbb{C})_{\rho_{V}})=1$ and $\dim_{\mathbb{C}} H^1(T^2; \mathfrak{sl}_2(\mathbb{C})_{\rho_{T^2}})=2$. So $\dim_{\mathbb{C}} H^1(M; \mathfrak{sl}_2(\mathbb{C})_{\rho_M})= 1$. By \cite[Proposition 3.5]{porti_french}, we have an inclusion of the Zariski tangent space $T^{\text{Zar}}_{\rho_M}(X_c(M))\hookrightarrow H^1(M; \mathfrak{sl}_2(\mathbb{C})_{\rho_M})$. So $\dim_{\mathbb{C}} T^{\text{Zar}}_{\rho_M}(X_c(M))\leq \dim_{\mathbb{C}}H^1(M; \mathfrak{sl}_2(\mathbb{C})_{\rho_M})= 1$.

Following from Thurston's result \cite[Proposition 3.2.1]{CS}, $\dim_{\mathbb{C}} X_c(M)\geq 1$ as $\rho_M(\text{im}(\pi_1(\partial M)\rightarrow \pi_1(M)))=\mathbb{Z}$.
Since $\dim_{\mathbb{C}} X_c(M)\leq \dim_{\mathbb{C}} T^{\text{Zar}}_{\rho_M}(X_c(M))$, then $\dim_{\mathbb{C}} X_c(M)=\dim T^{\text{Zar}}_{\rho_M}(X_c(M))=\dim_{\mathbb{C}} H^1(M; \mathfrak{sl}_2(\mathbb{C})_{\rho_M})= 1$. Therefore $[\rho_M]$ is a smooth point of $X(M)$.

To show that the Galois conjugate $\rho_0$ of $\rho_M$ is also a smooth point, we use the same argument as in the proof of \cite[Lemma 8.3]{CD16}. Construct $X_1$ by taking the $\mathbb{C}$-irreducible component $X_0$ of $X(M)$ containing $[\rho_M]$, which must be defined over some number field, and then take the union of the $Gal(\overline{\mathbb{Q}}/\mathbb{Q})$-orbit of $X_0$. Then $X_1$ is the unique $\mathbb{Q}$-irreducible component of $X(M)$ that contains $[\rho_M]$.  Since $X_1$ is invariant under the $Gal(\overline{\mathbb{Q}}/\mathbb{Q})$-action, it contains $[\rho_0]$ as well as $[\rho_M]$. As by definition,  $T^{\text{Zar}}_{\rho_M}(X(M))$ is defined by derivatives of a set of polynomials. Then $T^{\text{Zar}}_{\rho_0}(X(M))$ is defined by derivatives of Galois conjugates of this set of polynomials and thus should have dimension $1$, same as $T^{\text{Zar}}_{\rho_M}(X(M))$. On the other hand, any component of $X_1$ has the same dimension as $X_c(M)$, which is $1$. So $[\rho_0]$ is a smooth point of $X_1$ and thus of $X(M)$.

Moreover, By Th\'{e}or\`{e}me 3.15 of \cite{porti_french}, $\rho_M$ is $\gamma$-regular for some simple closed curve $\gamma\subset\partial M$, which means that the inclusion $H^1(M,\mu; \mathfrak{sl}_2(\mathbb{C})_{\rho_M})\rightarrow H^1(M; \mathfrak{sl}_2(\mathbb{C})_{\rho_M})$ is nonzero (see \cite[Definition 3.21]{porti_french} for definition). So $\text{tr}_{\gamma}$ is a local parameter of $X(M)$ at $[\rho_M]$. Since $[\rho_M]$ is not $\lambda$-regular as $\rho_M(\lambda)=I$, then $\gamma$ must be a curve different from $\lambda$. Locally the sign of $\text{tr}_{\gamma}$ does not change, so we could make $\text{tr}^2_{\gamma}$ the local parameter. Whether a regular function is a local parameter at a smooth point on the curve $X_1$ can be expressed purely algebraically and hence is $Gal(\overline{\mathbb{Q}}/\mathbb{Q})$-invariant. It follows that $[\rho_0]$ is also a smooth point of $X_1$ with local parameter $\text{tr}_{\gamma}^2$.

Applying \cite[Proposition 2.8]{CD16}, we get a smooth arc $\overline{c}$ of real points in $X_{\mathbb{R}}(M)$ containing $[\rho_0]$, locally defined by $\text{tr}^2_{\gamma}$ being real.
By restricting $\epsilon$ if necessary, we can assume that every character in $\overline{c}$ comes from an irreducible PSL$_2\mathbb{C}$ representation. Since $[\rho_0]\in X_{G}(M)$ is irreducible, we can restrict $\epsilon$ so that $\overline{c}$ is actually contained in $X_{G}(M)$ as both $X_{G}(M)$ and $X_{\text{SU}_2(\mathbb{C})}(M)$ are closed in $X(M) $\cite[Lemma 2.12]{CD16}.
Then by \cite[Lemma 2.11]{CD16} we can lift $\overline{c}$ to $c\in R_{G}(M)$ and $c$ is still parameterized by $\text{tr}^2_{\gamma}$. 
This completes the proof of (b).
\end{proof}

\begin{lemma}\label{trans}%[8.2]
trans$(\widetilde{\rho_0}(\lambda))$ is an even integer.
\end{lemma}

\begin{proof}%[Proof of Lemma \ref{trans}]
When mapping down to SL$_2\mathbb{R}$, the image of $\widetilde{\rho_0}(\lambda)\in \widetilde{\text{PSL}_2\mathbb{R}}$ is $I$. It follows from \cite[Claim 8.5]{CD16} that trans$(\widetilde{\rho_0}(\lambda))$ is an even integer.
\end{proof}

Now we are ready to prove Theorem \ref{Galois}.

\begin{proof}[Proof of Theorem \ref{Galois}]
The idea of proof is similar to that of Theorem \ref{Thm_Alex}: show that a nontrivial $\widetilde{G}$ representation of $\pi_1(M(r))$ exists for  rational $r$ in some small interval containing $0$. 

First we lift the arc $c\subset R_G(M)$ as constructed in Lemma \ref{galois_smooth} to $\tilde{c}\in R_{\widetilde{G}}(M)$. In the case of hyperbolic integer solid torus $M$, $H^2(\pi_1(M);\mathbb{Z})\cong H^2(M; \mathbb{Z}) = 0$, so we can always lift $G$ representations of $M$ to $\widetilde{G}$.

Since $M(0)$ admits a complete hyperbolic structure, elements in $\pi_1(M(0))$ are mapped to loxodromic elements in PSL$_2\mathbb{C}$ by $\rho_{hyp}$. So $\mu\in \pi_1(M(0))$ is mapped to either a hyperbolic or elliptic element under the Galois conjugate $\rho_0$. Therefore we divide our proof into two cases according to the image of the meridian $\mu$.

\begin{remark}
We do not consider the case that $\mu$ is mapped to a parabolic element, because $\rho_{\text{hyp}}(\mu)$ is loxodromic and Galois conjugate cannot take a complex number with norm greater than $2$ to one with norm $2$.
\end{remark}

\textbf{Case 1:} $\mu$ is mapped to an elliptic element by $\rho_0$.

At $\widetilde{\rho_0}$, the local parameter $s=\text{tr}^2(\widetilde{\rho_0}(\gamma))<4$. As $\tilde{c}$ is parameterized near $\widetilde{\rho_0}$ by $\text{tr}_{\gamma}^2 \in [s-\epsilon, s+\epsilon]$, we can require $s+\epsilon<4$ so that $\tilde{c}\subset PE_{\widetilde{G}}(M)$.
Then we map $\tilde{c}$ down to an arc $A\subset EL_{\widetilde{G}}(M)$ which is locally parameterized by $\text{tr}^2_{\gamma}$ on some small interval $[0,\delta]$.

To obtain an interval of orderable Dehn fillings, we want to apply Lemma 8.4 of \cite{CD16} which works similarly as Lemma \ref{main lemma}. So we need to show that $A$ is not contained in the horizontal axis $L_0$ of $EL_{\widetilde{G}}(M)\subset \mathbb{R}^2$. If it is contained in $L_0$, then all $\rho_t$ factor through $\pi_1(M(0))$ and it follows that $[\rho_t]$ lie in an irreducible component of $X(M(0))$ with complex dimension at least one. But we have seen in the proof of Lemma \ref{galois_smooth} that $H^1(M(0); \mathfrak{sl}_2(\mathbb{C})_{\rho_{hyp}})=0$, so $1\leq \dim T^{\text{Zar}}_{\rho_0}(X(M(0)))=\dim T^{\text{Zar}}_{\rho_{hyp}}(X(M(0)))\leq\dim_{\mathbb{C}} H^1(M(0); \mathfrak{sl}_2(\mathbb{C})_{\rho_{hyp}})= 0$, which is a contradiction.

Now we can draw an arc $A$ inside the translation extension locus $EL_{\widetilde{G}}(M)$ near $\widetilde{\rho_0}$. It contains no ideal point as all points on $A$ come from $\widetilde{G}$ representations. %in $\mathbb{R}^2$.
Applying Lemma 8.4 of \cite{CD16}, we get $a>0$ so that $L_r$ meets $EL_{\widetilde{G}}(M)$ for all $r$ in the interval $(-a,a)$. Invoking \cite[Theorem 1.2]{GL1} (at most three Dehn fillings of a knot complement are reducible), we can shrink $a$ to make $M(r)$ irreducible. Then we can apply Lemma 4.4 of \cite{CD16}. 

\textbf{Case 2:} $\mu$ is mapped to a hyperbolic element by $\rho_0$.

This case is similar to Case 1 except we start with s=tr$^2(\widetilde{\rho_0}(\gamma))>4$. As $\tilde{c}$ is parameterized by $\text{tr}_{\gamma}^2 \in [s-\epsilon, s+\epsilon]$, we can require $s-\epsilon>4$ so that $\tilde{c}\subset PH_{\widetilde{G}}(M)$.
Again we map $\tilde{c}$ down to an arc $A\subset HL_{\widetilde{G}}(M)$ which is locally parameterized by $\text{tr}^2_{\gamma}$ on some small interval $[-\delta,\delta]$.

By Lemma \ref{invariant}, we can always choose a lift $\widetilde{\rho_0}$ such that trans$(\widetilde{\rho_0}(\mu))=0$ and therefore by continuity we can make $A$ lie in $H_{0,j}(M)$. To show $A\subset H_{0,0}(M)$, we compute $j=\text{trans}(\widetilde{\rho_0}(\lambda))$ and show it is $0$. By assumption, $M(0)$ is a mapping torus of a homeomorphism of a genus $2$ surface $S$. Then $M(0)=M_{\phi}$ where $\phi$ is a Pseudo-Anosov map of $S$ since $M(0)$ is hyperbolic. Suppose there is a $G$ representation $\rho_0$ of $\pi_1(M(0))$, then it restricts to a $G$ representation $\rho_0|_S$ of $\pi_1(S)$.
Let $\text{eu}(\rho_0|_S)$ be the Euler number of $\rho_0|_S$ as defined in \cite{Goldman} (or equivalently in \cite{Milnor, Wood}). It is equal to trans$(\widetilde{\rho_0}([a_1, b_1][a_2, b_2]))$ with $a_1, b_1, a_2, b_2$ the standard generators of $\pi_1(S)$ and is thus equal to trans$(\widetilde{\rho_0}(\lambda))$.
We claim that $|\text{eu}(\rho_0|_S)|\neq 2$. Otherwise $\rho_0|_S$ would determine a hyperbolic structure on $S$ (Milnor-Wood inequality \cite{Milnor, Wood}) which is invariant under $\phi$, implying that $\phi$ has finite order which contradicts that $\phi$ is Pseudo-Anosov. So $|\text{trans}(\widetilde{\rho_0}(\lambda))|=|\text{eu}(\rho_0|_S)|\neq 2$. By Lemma \ref{trans} and Proposition \ref{bounded}, we must have $\text{trans}(\widetilde{\rho_0}(\lambda))=0$.

Claim that $A$ is not contained in the horizontal axis $L_0$ of $H_{0,0}\subset\mathbb{R}^2$. If it is contained in the horizontal axis, then all $\rho_t$ factor through $\pi_1(M(0))$ and it follows that $[\rho_t]$ lie in an irreducible component of $X(M(0))$ with complex dimension at least one. But we have seen that $H^1(M(0); \mathfrak{sl}_2(\mathbb{C})_{\rho_{hyp}})=0$, so $1\leq \dim T^{\text{Zar}}_{\rho_0}(X(M(0)))= \dim T^{\text{Zar}}_{\rho_{hyp}}(X(M(0))) \leq\dim_{\mathbb{C}} H^1(M(0); \mathfrak{sl}_2(\mathbb{C})_{\rho_{hyp}})= 0$, which is a contradiction.

So we have constructed an arc $A\subset H_{0,0}(M)$ that is not contained in $L_0$ near $\widetilde{\rho_0}$. Then we can find $a>0$ such that $L_r$ meets $H_{0, 0}(M)$ at points that are not parabolic or ideal and $M(r)$ irreducible for all $r$ in an interval $(0, a)$ or $(-a, 0)$. Applying Lemma \ref{main lemma} then tells us $M(r)$ is orderable for $r$ in $(0, a)$ or $(-a, 0)$.

Finally, we show $M(0)$ is orderable. The first Betti number of $M(0)$ is $1$ as the integral homology groups of $M(0)$ are the same as those of $S^2\times S^1$. The irreducibility of $M(0)$ follows from the assumption that it is hyperbolic. So we can apply Theorem 1.1 of \cite{BRW} and show that $\pi_1(M(0))$ is left-orderable, completing the proof of the theorem.

\end{proof}

\begin{remark}
The assumption that $M(0)$ being a mapping torus of genus $2$ is used to show $\text{trans}(\lambda)=0$. This is a very strong hypothesis. However, when $\mu$ is mapped to an elliptic element, $M(0)$ being a mapping torus is not needed at all. When $\mu$ is mapped to hyperbolic, the author does not know how to prove the theorem under a weakened assumption.
\end{remark}

Using the method of Calegari \cite[Section 3.5]{real_place}, we are able to prove the following result.
\begin{proposition}
Suppose $Y$ is a mapping torus of a closed surface $S$ of genus at least $2$. If $Y$ has a faithful $G$ representation $\rho$, then $\rho|_S$ can never be discrete.
\end{proposition}
\begin{proof}
First notice that $\pi_1(Y)$ has no torsion. This is because $Y$ is an Eilenberg-Maclane space and a finite dimensional CW-complex as a mapping torus. We claim that $\rho$ has indiscrete image. Otherwise $\rho(\pi_1(Y))\leq G=\text{PSL}_2(\mathbb{R})=\text{Isom}^+(\mathbb{H}^2)$ would be a torsion-free Fuchsian group and thus act on $\mathbb{H}^2$ with quotient isometric to a complete hyperbolic surface, which is impossible as $Y$ is a closed $3$ manifold.

Now suppose $\rho|_S$ is discrete, then $\rho|_S: \pi_1(S) \rightarrow G$ determines some complete hyperbolic structure $\mathbb{H}^2/\rho(\pi_1(S))$ on $S$ as it is faithful. So $\rho(\pi_1(S))$ consists of hyperbolic elements only.  %Isom$^+(S)$
Let $\pi_1(Y)=\left<t\right>\ltimes \pi_1(S)$, where the conjugation action of $t$ on $\pi_1(S)$ is given by the monodromy of the bundle. Then $\rho(t)$ acts on $\rho(\pi_1(S))$ by conjugation and normalizes $\rho(\pi_1(S))$. 
Since $\rho(t) \in G$, it gives an isometry of $\mathbb{H}^2$ and thus an isometry of $S$. On the other hand, any isometry of $S$ is of finite order as it has to preserve the hyperbolic structure. Then the action of $\rho(t)$ on $\rho(\pi_1(S))$ by conjugation must be of finite order. To show that actually $\rho(t)$ is a finite order element in $G$, notice that $\rho(\pi_1(S))$ has at least two hyperbolic elements of different axes. But this contradicts the fact that $\rho$ is a faithful representation as $\pi_1(Y)$ has no torsion. So $\rho|_S$ could not be discrete.
\end{proof}

We are also able to prove the following result.
\begin{proposition}
Suppose $M$ is a hyperbolic $\mathbb{Z}$ homology solid torus. Assume the longitudinal filling $M(0)$ is hyperbolic and the trace field of its holonomy representation has a real embedding at which the associated quaternion algebra splits. Then the $\mathbb{Z}$HS Dehn filling $M(\frac{1}{n})$ is orderable for all $n\in \mathbb{Z}$ large enough (or $-n$ large enough).
\end{proposition}

\begin{proof}
The proof is almost the same as Theorem \ref{Galois} except for the case when the meridian $\mu$ is mapped to a hyperbolic element by $\rho_0$. Similar to Case 2 of the proof of Theorem \ref{Galois}, first we construct $A\subset H_{i,j}(M)$ and show that it is not contained in the horizontal axis $L_0$. But after that, we do not need to show $i=j=0$ (i.e. $A\subset H_{0,0}(M)$). 
Since $A$ is not horizontal, there exists $N\in \mathbb{Z}_{>0}$ (or $N\in \mathbb{Z}_{<0}$) large enough (or $-N$ large enough resp.) such that $L_{-\frac{1}{N}}$ intersects $A$  at points that are not parabolic or ideal and $M(\frac{1}{n})$ irreducible for all $n\geq N$ (or $-n\geq -N$ resp.). Suppose $\rho_n$ corresponds to an intersection point of $L_{-\frac{1}{n}}$ and $A$, then $\rho_n(\mu\lambda^n)=I$.
By choosing the lift $\widetilde{\rho_n}$ of $\rho_n$, we can make trans $\widetilde{\rho_n}(\mu)=-n\cdot$ trans $\widetilde{\rho_n}(\lambda)$, which then implies trans $\widetilde{\rho_n}(\mu\lambda^n)=0$. So we actually have $\widetilde{\rho_n}(\mu\lambda^n)=I$ and therefore $\widetilde{\rho_n}$ is a nontrivial $\widetilde{G}$ representation of $\pi_1(M(\frac{1}{n}))$.
\end{proof}

\section{Unsolved Problems} \label{problems}
Here are some interesting questions for potential follow-up researches:
\begin{enumerate}
\item[1)] Can we drop the longitudinal rigid condition in Theorem \ref{Thm_Alex}? In particular, is it possible prove $H^1(\pi_1(M(0)); \mathfrak{sl}_2(\mathbb{R})_{\rho^+}))=0$, which is weaker than $M$ being longitudinal rigid but enough to prove Theorem \ref{Thm_Alex}?

\item[2)] As mentioned in the remark after the proof of Theorem \ref{Thm_Alex}, according to our numerical experiment, a larger range of slopes of orderable Dehn filling is expected. But unfortunately the author was not able to prove it. Is it possible to extend the interval $(-a,0]$ in Theorem \ref{Thm_Alex} to $(-a, b)$ with $a, b > 0$ using some properties of the character variety?

\item[3)] In Theorem \ref{Galois}, we assumed that the holonomy representation has a trace field with a real embedding. When do holonomy representations have real places? Calegari studied some special cases in \cite{real_place}. Are there more general criteria?

\item[4)] In Theorem \ref{Galois}, we also assumed the $0$ filling on $M$ is a mapping torus of genus $2$. This is because we need the translation number of the homological longitude of $M$ to be $0$. It is in general a challenging question to compute the translation number. Is there an algorithm to compute the translation number of the longitude of $M$? Can we weaken the restriction on the genus and still have the translation number of the longitude being $0$?

\end{enumerate}

\bibliographystyle{mrl}
\bibliography{reference_list}

\end{document}